\documentclass[11pt,reqno]{amsart}
\usepackage{fullpage}

\usepackage{graphicx}
\usepackage[draft]{hyperref}
\usepackage{amsmath,amsopn,amssymb,amsfonts,stmaryrd}
\usepackage{verbatim}
\usepackage{amsthm}
\usepackage{mathtools}
\usepackage{color}
\usepackage{enumitem}
\usepackage[framemethod=TikZ]{mdframed}
\usepackage{bbm}
\usepackage{mathrsfs}
\usepackage{booktabs}
\usepackage{caption}
\usepackage{bm}

\usepackage{tensor}
\usepackage{xcolor}
\usepackage{bbm}
\usepackage{cleveref}

\makeatletter
\@namedef{subjclassname@2020}{%
 \textup{2020} Mathematics Subject Classification}
\makeatother

\newcommand{\IR}{{\mathbb R}}%Reals
\newcommand{\IC}{{\mathbb C}}%Complex
\newcommand{\E}{\mathbb{E}}
\newcommand{\CC}{\mathcal{C}}

\newcommand{\CH}{\mathcal{H}}

\newcommand{\CK}{\mathcal{K}}
\newcommand{\CL}{\mathcal{L}}
\newcommand{\CM}{\mathcal{M}}
\newcommand{\CN}{\mathcal{N}}

\newcommand{\CP}{\mathcal{P}}
\newcommand{\CR}{\mathcal{R}}
\newcommand{\CS}{\mathcal{S}}

\newcommand{\1}{\mathbbm{1}}

\newcommand{\SU}{\mathfrak{su}}
\newcommand{\SO}{\mathfrak{so}}

\newcommand{\Arg}{\operatorname{Arg}}
\newcommand{\pp}{\operatorname{pp}}

\newcommand{\so}[1]{{\mathfrak{so}{#1}}}

\newcommand{\nor}[1]{||#1||}

\newcommand{\N}{\mathbb N}
\newcommand{\C}{\mathbb C}

\theoremstyle{plain}
\newtheorem{thm}{Theorem}[section]
\newtheorem{cor}[thm]{Corollary}
\newtheorem{lem}[thm]{Lemma}
\newtheorem{prop}[thm]{Proposition}

\theoremstyle{definition}

\newtheorem*{Defn}{Definition}

\numberwithin{equation}{section}

\theoremstyle{theorem}
\newtheorem*{rems}{Remarks}

\newcommand{\ceil}[1]{\lceil #1\rceil}

\newcommand{\flo}[1]{\lfloor #1\rfloor}
\newcommand{\pflo}[1]{\left\lfloor #1\right\rfloor}
\newcommand{\sm}{\setminus}
\newcommand{\bs}{\backslash}

\def\a{\alpha}
\def\b{\beta}
\def\d{\delta}

\def\k{\kappa}
\def\l{\lambda}
\def\p{\rho}

\def\w{\omega}

\def\z{\zeta}
\def\th{\theta}

\def\e{\varepsilon}
\def\vf{\varphi}

\def\s{\sigma}
\def\g{\gamma}
\def\t{\tau}

\def\a{\alpha}
\def\b{\beta}
\def\d{\delta}

\def\k{\kappa}
\def\l{\lambda}
\def\p{\varrho}

\def\w{\omega}

\def\z{\zeta}
\def\th{\theta}

\def\e{\varepsilon}
\def\vf{\varphi}

\def\s{\sigma}
\def\g{\gamma}
\def\t{\tau}

\def\GG{\Gamma}

\def\LL{\Lambda}

\def\del{ \partial}

\newcommand{\re}{{\rm Re}}
\newcommand{\im}{{\rm Im}}

\newcommand{\R}{\mathbb R}
\newcommand{\Z}{\mathbb Z}

\newcommand{\Log}{\mathrm{Log}}

\newcommand{\us}[2]{\underset{#1}{#2}}

\newcommand{\pdd}[2]{\frac{d^#2}{d#1^#2}}

\newcommand{\MT}{\operatorname{MT}}

\setlist[itemize]{noitemsep, topsep=0pt}

\makeatletter
\newcommand{\vast}{\bBigg@{2}}
\newcommand{\Vast}{\bBigg@{5}}
\makeatother

\newcommand{\Res}{\operatorname{Res}}

\definecolor{Green}{rgb}{0,0.4,0}

\allowdisplaybreaks

\title{Asymptotic expansions for partitions generated by infinite products} 
\author{Walter Bridges}
\author{Benjamin Brindle}
\author{Kathrin Bringmann}
\address{University of Cologne, Department of Mathematics and Computer Science, Weyertal 86-90, 50931 Cologne, Germany}
\email{wbridges@uni-koeln.de}
\email{bbrindle@uni-koeln.de}
\email{kbringma@math.uni-koeln.de}
\email{jfrank12@uni-koeln.de}
\author{Johann Franke}


\begin{document}
\maketitle

\begin{abstract}
	Recently, Debruyne and Tenenbaum proved asymptotic formulas for the number of partitions with parts in $\LL\subset\N$ ($\gcd(\LL)=1$) and good analytic properties of the corresponding zeta function, generalizing work of Meinardus. In this paper, we extend their work to prove asymptotic formulas if $\LL$ is a multiset of integers and the zeta function has multiple poles. In particular, our results imply an asymptotic formula for the number of irreducible representations of degree $n$ of $\so{(5)}$. We also study the Witten zeta function $\z_{\so{(5)}}$, which is of independent interest.
\end{abstract}

\section{Introduction and statement of results}

\subsection{The Circle Method}

In analytic number theory and combinatorics, one uses complex analysis to better understand properties of sequences. Suppose that a sequence $(c(n))_{n\in\N_0}$ has moderate growth and the {\it generating function}
$$
	f(q):=\sum_{n \geq 0} c(n)q^n,
$$
is holomorphic in the unit disk with radius of convergence 1. Via Cauchy's integral formula one can then recover the coefficients from the generating function
\begin{align} \label{paper:eq:Cauchy}
	c(n) = \frac{1}{2\pi i} \int_{\mathcal{C}} 	\frac{f(q)}{q^{n+1}} dq,
\end{align}
for any closed curve $\CC$ contained in the unit disk that surrounds the origin exactly once counterclockwise. The so-called Circle Method uses the analytic behavior of $f(q)$ near the boundary of the unit circle to obtain asymptotic information about $c(n)$. For instance, if the $c(n)$ are positive and monotonically increasing, it is expected that the part close to $q=1$ provides the dominant contribution to \eqref{paper:eq:Cauchy}. These parts of the curve are the \textit{major arcs} and the complement are the \textit{minor arcs}. To obtain an asymptotic expansion for $c(n)$, one then evaluates the major arc to some degree of accuracy and bounds the minor arcs. Depending on the function $f(q)$, both of these tasks present a variety of difficulties.

In the present paper, we are interested in infinite product generating functions of the form 
$$
	f(q) = \prod_{n \geq 1} \frac{1}{(1 - q^n)^{a(n)}}.
$$
Such generating functions are important in the theory of partitions, but also arise, for example, in representation theory. If $a(n)$ is a ``simple'' sequence of nonnegative integers and $f$ is ``bounded'' away from $q=1$, then Meinardus \cite{Meinardus} proved an asymptotic expression for $c(n)$. Debruyne and Tenenbaum \cite{DebrTen} eliminated the technical growth conditions on $f$ by adding a few more assumptions on the $a(n)$, which made their result more applicable. Our main results, Theorems \ref{paper:T:main2} and \ref{paper:T:TwoPoleAsymptotics}, yield asymptotic expansions given mild assumptions on $a(n)$ and have a variety of new applications.

\subsection{The classical partition function}

Let $n \in \N$. A weakly decreasing sequence of positive integers that sum to $n$ is called a {\it partition} of $n$. The number of partitions is denoted by $p(n)$. If $\l_1+\ldots+\l_r=n$, then the $\l_j$ are called the {\it parts} of the partition. The partition function has no elementary closed formula, nor does it satisfy any finite order recurrence. However, setting $p(0):=1$, its generating function has the following product expansion
\begin{equation}\label{paper:partitionproduct}
	\sum_{n\ge0} p(n)q^n = \prod_{n\ge1} \frac{1}{1-q^n},
\end{equation}
where $|q| < 1$. In \cite{HardyRama}, Hardy and Ramanujan used \eqref{paper:partitionproduct} to show the asymptotic formula 
\begin{equation}\label{paper:partitionasy}
	p(n) \sim \frac{1}{4\sqrt{3}n}e^{\pi\sqrt{\frac{2n}{3}}},\qquad n\to\infty,
\end{equation}
which gave birth of the Circle Method. With Theorem \ref{paper:T:main2} we find, for certain constants $B_j$ and arbitrarily $N \in \N$,
\begin{equation*}%\label{paper:partitionasy2}
	p(n) = \frac{e^{\pi\sqrt{\frac{2n}{3}}}}{4\sqrt{3}n}\left(1+\sum_{j=1}^N \frac{B_j}{n^\frac j2}+O_N\left(n^{-\frac{N+1}{2}}\right)\right).
\end{equation*}
Similarly, one can treat the cases for $k$-th powers (in arithmetic progressions), see \cite{DebrTen}.

\subsection{Plane partitions}\label{paper:sect:Plane-partitions}

Another application is an asymptotic formula for plane partitions. A {\it plane partition of size $n$} is a two-dimensional array of non-negative integers $\pi_{j,k}$ for which $\sum_{j,k}\pi_{j,k}=n$, such that $\pi_{j,k}\ge\pi_{j,k+1}$ and $\pi_{j,k}\ge\pi_{j+1,k}$ for all $j,k\in\N$. We denote the number of plane partitions of $n$ by $\pp(n)$. MacMahon \cite{MacMahon} proved that
\[
	\sum_{n\ge0} \pp(n)q^n = \prod_{n\ge1} \frac{1}{\left(1-q^n\right)^n}.
\]
Using Theorem \ref{paper:T:main2}, we recover Wright's asymptotic formula \cite{Wright}
\begin{equation*}%\label{paper:eq:planeasy}
	\pp(n) = \frac{C}{n^\frac{25}{36}} e^{A_1 n^\frac23}\left(1+\sum_{j=2}^{N+1} \frac{B_j}{n^\frac{2(j-1)}{3}}+O_N\left(n^{-\frac{2(N+1)}{3}}\right)\right),
\end{equation*}
where the constants $B_j$ are explicitly computable, 
$$
	C:=\frac{\z(3)^\frac{7}{36}e^{\z'(-1)}}{2^{\frac{11}{36}}\sqrt{3\pi}}, \qquad A_1:=\frac{3\z(3)^\frac13}{2^\frac23}
$$
with $\z$ the Riemann zeta function.

\subsection{Partitions into polygonal numbers}

The $n$-th {\it$k$-gonal number} is given by ($k\in\N_{\ge3}$)
\begin{align*}
	%P_k(n) := \frac{n^2(k-2) - n(k-4)}{2}.
	P_k(n) := \frac12 \left((k-2)n^2 +(4-k)n\right).
\end{align*}
The study of representations of integers as sums of polygonal numbers has a long history. Fermat conjectured in 1638 that every $n\in \N$ may be written as the sum of at most $k$ $k$-gonal numbers which was finally proved by Cauchy. Let $p_k(n)$ denotes the number of partitions of $n$ into $k$-gonal numbers. We have the generating function
$$
	\sum_{n \geq 0} p_k(n)q^n = \prod_{n \geq 1} \frac{1}{1-q^{P_k(n)}}.
$$
The $p_k(n)$ have the following asymptotics.\footnote{Note that asymptotics for polynomial partitions were investigated in a more general setting by Dunn and Robles in \cite{DunnRob}.}

\begin{thm}\label{paper:T:n-gonalpartitions}
	We have, for all\ \footnote{Explicit asymptotic formulas for $p_3(n)$, $p_4(n)$, and $p_5(n)$ are given in Corollary \ref{paper:cor:n-gonalpartitions}.} $N\in\N$,
	\[
		p_k(n) = \frac{C(k)e^{A(k)n^\frac13}}{n^\frac{5k-6}{6(k-2)}}\left(1+\sum_{j=1}^N \frac{B_{j,k}}{n^\frac j3}+O_N\left(n^{-\frac{N+1}{3}}\right)\right),
	\]
	where the $B_{j,k}$ can be computed explicitly and 
	\[
		C(k) := \frac{(k-2)^\frac{6-k}{6(k-2)}\GG\left(\frac{2}{k-2}\right)\z\left(\frac32\right)^\frac{k}{3(k-2)}}{2^{\frac{3k-2}{2(k-2)}}\sqrt{3}\pi^\frac{4k-9}{3(k-2)}},\qquad A(k) := \frac32\left(\sqrt{\frac{\pi}{k-2}}\z\left(\frac32\right)\right)^\frac23.
	\]
\end{thm}

\subsection{Numbers of finite-dimensional representations of Lie algebras}

The special unitary group $\SU(2)$ has (up to equivalence) one irreducible representation $V_k$ of each dimension $k\in \N$. Each $n$-dimensional representation $\bigoplus_{k=1}^\infty r_k V_k$ corresponds to a unique partition 
\begin{align}
	\label{paper:partition} n = \lambda_1 + \lambda_2 + \cdots + \lambda_r, \qquad \lambda_1 \geq \lambda_2 \geq \ldots \geq \lambda_r \geq 1
\end{align}
such that $r_k$ counts the number of $k$ in \eqref{paper:partition}. As a result, the number of representations equals $p(n)$. It is natural to ask whether this can be generalized. The next case is the unitary group $\SU(3)$, whose irreducible representations $W_{j,k}$ indexed by pairs of positive integers. Note that (see Chapter 5 of \cite{Hall}) $\dim(W_{j,k})=\frac12jk(j+k)$. Like in the case of $\SU(2)$, a general $n$-dimensional representation decomposes into a sum of these $W_{j,k}$, again each with some multiplicity. So analogous to \eqref{paper:partitionproduct}, the numbers $r_{\SU(3)}(n)$ of $n$-dimensional representations, have the generating function
\begin{equation*}%\label{paper:G}
	\sum_{n\ge0} r_{\SU(3)}(n)q^n = \prod_{j,k\ge1} \frac{1}{1-q^{\frac{jk(j+k)}{2}}},
\end{equation*}
again with $r_{\SU(3)}(0):=1$. In \cite{Ro}, Romik proved that, as $n\to\infty$,
\begin{equation*}%\label{paper:rn}
	r_{\SU(3)}(n) \sim \frac{C_0}{n^{\frac 35}} \exp\left(A_1n^{\frac 25}+A_2n^{\frac{3}{10}}+A_3n^{\frac 15}+A_4n^{\frac{1}{10}}\right),
\end{equation*}
with explicit constants\footnote{Note that Romik used different signs for the constants in the exponential.} $C_0,A_1,\dots,A_4$ expressible in terms of zeta and gamma values. Two of the authors \cite{BF} improved this to an analogue of formula \eqref{paper:partitionasy}, namely, for any $N\in\N_0$, we have
\begin{align} \label{paper:eq:rsu3asy}
	r_{\SU(3)}(n) = \frac{C_0}{n^{\frac 35}}\exp\left(A_1n^{\frac 25}+A_2n^{\frac{3}{10}}+A_3n^{\frac 15}+A_4n^{\frac{1}{10}}\right)\left(1+\sum_{j=1}^{N}\frac{C_j}{n^\frac{j}{10}} + O_N\left( n^{-\frac{N}{10}-\frac{3}{80}}\right)\right),
\end{align}
as $n\to\infty$, where the constants $C_j$ do not depend on $N$ and $n$ and can be calculated explicitly. The expansion \eqref{paper:eq:rsu3asy} with improved error term $O_N(n^{-\frac{N+1}{10}})$ and explicit values for $A_j$ ($1\le j\le4$) and $C_0$, can also be obtained using \Cref{paper:T:TwoPoleAsymptotics}.

This framework generalizes to other groups. For example, one can investigate the {\it Witten zeta function} for $\SO(5)$, which is (for more background to this function, see \cite{Mat} and \cite{MatTsu})
\begin{equation}\label{paper:eq:SO(5)zeta}
 \z_{\SO(5)}(s) := \sum_\varphi \frac1{\dim(\varphi)^{s}} = 6^s\sum_{n,m\ge1}\frac1{ m^{s}n^{s}(m+n)^{s}(m+2n)^{s}},
\end{equation}
where the $\vf$ are running through the finite-dimensional irreducible representations of $\SO(5)$. We prove the following; for the more precise statement see Theorem \ref{paper:T:mainSO5}.

\begin{thm}\label{paper:T:shortermainSO5}
	The function $\z_{\so(5)}$ has a meromorphic continuation to $\C$ whose positive poles are simple and occur for $s\in\{\frac12,\frac13\}$.
\end{thm}

It is well-known that the finite-dimensional representations of $\SO(5)$ can be doubly indexed as $(\vf_{j,k})_{j,k\in\N}$ with $\dim(\vf_{j,k})=\frac16jk(j+k)(j+2k)$, which explains the last equality in \eqref{paper:eq:SO(5)zeta}. A general $n$-dimensional representation decomposes as a sum of these $\vf_{j,k}$, each with some multiplicity. Therefore, as in the case $\SU(3)$, we find that
\begin{align*}
	\sum_{n \geq 0} r_{\SO(5)}(n)q^n = \prod_{j,k \geq 1} \frac{1}{1-q^{\frac{jk(j+k)(j+2k)}{6}}}.
\end{align*}
We prove the following.
 
\begin{thm}\label{paper:T:rsu5asy}
	As $n\to\infty$, we have, for any $N\in\N$,
	\begin{align*}
		r_{\SO(5)}(n) = \frac{C}{n^\frac{7}{12}}\exp\left( A_1 n^{\frac{1}{3}} + A_2 n^{\frac{2}{9}} + A_3 n^{\frac{1}{9}} + A_4 \right)\left( 1 + \sum_{j=2}^{N+1} \frac{B_j}{n^{\frac{j-1}{9}}} + O_N\left({n^{-\frac{N+1}{9}}}\right)\right),
	\end{align*}
	where $C$, $A_1$, $A_2$, $A_3$, and $A_4$ are given in \eqref{paper:A}--\eqref{paper:A2} and the $B_j$ can be calculated explicitly.
\end{thm}

\subsection{Statement of results}

The main goal of this paper is to prove asymptotic formulas for a general class of partition functions. To state it, let $f:\N\to\N_0$, set $\LL:=\N\setminus f^{-1}(\{0\})$, and for $q=e^{-z}$ ($z\in\C$ with $\re(z) > 0$), define 
\begin{align} \label{def:GfLf}
	G_f(z):=\sum_{n\ge0} p_f(n)q^n = \prod_{n\ge1} \frac{1}{\left(1-q^n\right)^{f(n)}},\qquad L_f(s) := \sum_{n\ge1} \frac{f(n)}{n^s}.
\end{align}
We require the following key properties of these objects.
\begin{enumerate}[leftmargin=*,label=\rm{(P\arabic*)}]
	\item\label{paper:main:1} Let $\a>0$ be the largest pole of $L_f$. There exists $L\in\N$, such that for all primes $p$, we have $|\LL\sm(p\N\cap\LL)|\ge L>\frac\a2$.
	
	%\item\label{paper:main:2} For all primes $p$, we have $|\LL\setminus(\LL\cap p\N)|=\infty$.
	
	\item\label{paper:main:3} Condition \ref{paper:main:3} is attached to $R\in\R^+$. The series $L_f(s)$ converges for some $s\in\C$, has a meromorphic continuation to $\{s\in\C:\re(s)\ge-R\}$, and is holomorphic on the line $\{s\in\C:\re(s)=-R\}$. The function $L_f^*(s):=\GG(s)\z(s+1)L_f(s)$ has only real poles $0<\a:=\g_1>\g_2>\dots$ that are all simple, except the possible pole at $s=0$, that may be double.
	
	\item\label{paper:main:4} For some $a<\frac\pi2$, in every strip $\s_1\le\s\le\s_2$ in the domain of holomorphicity, we uniformly have, for $s=\s+it$,
	\begin{align*}
		L_f(s) = O_{\sigma_1, \sigma_2}\left(e^{a|t|}\right), \qquad |t| \to \infty.
	\end{align*}
\end{enumerate}
Note that \ref{paper:main:1} implies that $|\LL\setminus(b\N\cap\LL)|\ge L>\frac\a2$ for all $b\ge2$. 
 
\begin{thm}\label{paper:T:main2}
	Assume \ref{paper:main:1} for $L\in\N$, \ref{paper:main:3} for $R>0$, and \ref{paper:main:4}. Then, for some $M,N\in\N$,
	\begin{equation*}%\label{paper:eq:pfasymptotic2}
		p_f(n) = \frac{C}{n^b}%\left(1+\sum_{j=1}^{\textcolor{red}{\bf ?}} \frac{B_j}{n^{\b_j}}+O_{R, L}\left(n^{-\b_{{\textcolor{red}{\bf ?}}+1}}\right)\right)
		\exp\left(A_1n^\frac{\a}{\a+1}+\sum_{j=2}^M A_jn^{\a_j}\right)\left(1+\sum\limits_{j=2}^N \frac{B_j}{n^{\b_j}} + O_{L,R}\left(n^{-\min\left\{\frac{2L-\a}{2(\a+1)},\frac{R}{\a+1}\right\}}\right)\right),
	\end{equation*}
	where $0\le\a_M<\a_{M-1}<\cdots\a_2<\a_1=\frac{\a}{\a+1}$ are given by\footnote{We can enlarge the discrete exponent sets at will, since we can always add trivial powers with vanishing coefficients to an expansion. Therefore, from now on we always use this expression, even if the set increases tacitly.} $\CL$ (defined in \eqref{paper:setL}), and $0<\b_2<\b_3<\dots$ are given by $\CM+\CN$, where $\CM$ and $\CN$ are defined in \eqref{paper:setM} and \eqref{paper:setN}, respectively. The coefficients $A_j$ and $B_j$ can be calculated explicitly; the constants $A_1$, $C$, and $b$ are provided in \eqref{paper:eq:mainconstants2} and \eqref{paper:eq:b}. Moreover, if $\a$ is the only positive pole of $L_f$, then we have $M=1$.
\end{thm}

\begin{rems}
	\ \begin{enumerate}[label=\textnormal{(\arabic*)},leftmargin=*]
		\item Debruyne and Tenenbaum proved Theorem \ref{paper:T:main2} in the special case that $f$ is the indicator function of a subset $\LL$ of $\N$. They also assumed that the associated $L$-function can be analytically continued except for one pole in $0<\a\le1$. Our refined assumption (P1) on the set $\LL$ is necessary to bound minor arcs in this more general setup.
		
		\item The complexity of the exponential term depends on the number and positions of the positive poles of $L_f$. \Cref{paper:T:TwoPoleAsymptotics} is more explicit and covers the case of exactly two positive poles. This case has importance for representation numbers of $\SU(3)$ and $\SO(5)$.
	\end{enumerate}
\end{rems}

In \Cref{sec:pre}, we collect some analytic tools, properties of special functions and useful properties of asymptotic expansions that are heavily used throughout the paper. In \Cref{sec:minmajarc}, we apply the Circle Method and calculate asymptotic expansions for the saddle point $\p_n$ and the value of the generating function $G_f(\p_n)$. In \Cref{sec:proof1.4}, we complete the proof of \Cref{paper:T:main2}, and we also state and prove a more explicit version of \Cref{paper:T:main2} in the case that $L_f$ has two positive poles (\Cref{paper:T:TwoPoleAsymptotics}). The proofs of Theorems \ref{paper:T:n-gonalpartitions}, \ref{paper:T:shortermainSO5}, and \ref{paper:T:rsu5asy} are given in \Cref{sec:proofs1.1-1.3}; this includes a detailed study of the Witten zeta function $\z_{\SO(5)}$ which is of independent interest.

\section*{Acknowledgements}

We thank Gregory Debruyne, Kohji Matsumoto, and Andreas Mono for helpful discussions. The first author and the third author were partially supported by the SFB/TRR 191 ``Symplectic Structure in Geometry, Algebra and Dynamics'', funded by the DFG (Projektnummer 281071066 TRR 191). The second and third author received funding from the European Research Council (ERC) under the European Union’s Horizon 2020 research and innovation programme (grant agreement No. 101001179), and the last two authors are partially supported by the Alfried Krupp prize.

\section*{Notation}

For $\b\in\IR$, we denote by $\{\b\}:=\b-\flo{\b}$ the {\it fractional part} of $\b$. As usual, we set $\mathbb{H}:=\{\t\in\C:\im(\t)>0\}$ and $\E:=\{z\in\C:|z|<1\}$. For $\d>0$, we define
\begin{align*}
	\mathcal{C}_\delta := \left\{z \in \C \colon |\mathrm{Arg}(z)| \leq \tfrac{\pi}{2} - \delta \right\},
\end{align*}
where $\Arg$ uses the principal branch of the complex argument. For $r>0$ and $z\in\C$, we set
$$
	B_r(z) := \{ w \in \IC : |w - z| < r\}.
$$
For $a,b\in\R$, we let $\CR_{a,b;K}$ be the rectangle with vertices $a\pm iK$ and $b\pm iK$, and we let $\del\CR_{a,b;K}$ be the path along the boundary of $\CR_{a,b;K}$, surrounded once counterclockwise. For $-\infty\le a<b\le\infty$, we denote $S_{a,b}:=\{z\in\C:a<\re(z)<b\}$. We also set, for real $\s_1\le\s_2$ and $\d>0$,
\[
	S_{\s_1,\s_2,\d} := \{s\in\C : \s_1\le\re(s)\le\s_2\} \Bigg\bs \left(B_\d\left(\frac12\right)\cup\bigcup_{j=-\infty}^1 B_\d\left(\frac j3\right)\right).
\]
For $k\in\N$ and $s\in\C$, the {\it falling factorial} is $(s)_k:=s(s-1)\cdots(s-k+1)$. For $f:\N\to\N_0$, we let $\CP$ be the set of poles of $L^*_f$, and for $R>0$ we denote by $\CP_R$ %$:=\CP_{f,R}$ 
the union of the poles of $L_f^*$ greater than $-R$ with $\{0\}$. We define
\begin{align}\label{paper:setL}
	\CL &:= \frac{1}{\a+1}\CP_R + \sum_{\mu\in\CP_R} \left(\frac{\mu+1}{\a+1}-1\right)\N_0,\\
	\label{paper:setM}
	\CM &:= \frac{\a}{\a+1}\N_0 + \left(-\sum_{\mu\in\CP_R} \left(\frac{\mu+1}{\a+1}-1\right)\N_0\right) \cap \left[0,\frac{R+\a}{\a+1}\right),\\
	\label{paper:setN}
	\CN &:= \left\{\sum_{j=1}^K b_j\th_j : b_j,K\in\N_0,\th_j\in(-\CL)\cap\left(0,\frac{R}{\a+1}\right)\right\}.
\end{align}
We set, with $\w_\a:=\Res_{s=\a}L_f(s)$,
\begin{align}\label{paper:eq:mainconstants2}
	A_1 &:= \left(1+\frac1\a\right)(\w_\a\GG(\a+1)\z(\a+1))^\frac{1}{\a+1},\qquad C := \frac{e^{L_f'(0)}(\w_\a\GG(\a+1)\z(\a+1))^\frac{\frac12-L_f(0)}{\a+1}}{\sqrt{2\pi(\a+1)}},\\
	\label{paper:eq:b}
	b &:= \frac{1-L_f(0)+\frac\a2}{\a+1}.
\end{align}

\section{Preliminaries}\label{sec:pre}

In this section, we collect and prove some tools used in this paper.

\subsection{Tools from complex analysis}

We require the following results from complex analysis. The first theorem describes Taylor coefficients of the inverse of a biholomorphic function; for a proof, see Corollary 11.2 on p. 437 of \cite{Char}.

\begin{prop}\label{paper:T:InverseFormula}
	Let $\phi:B_r(0)\to D$ be a holomorphic function such that $\phi(0)=0$ and $\phi'(0)\ne0$, with $\phi(z)=:\sum_{n\ge1}a_nz^n$. Then $\phi$ is locally biholomorphic and its local inverse of $\phi$ has a power series expansion $\phi^{-1}(w)=:\sum_{k\ge1}b_kw^k$, where
	\begin{equation*}%\label{paper:inverseformula}
		b_k = \frac{1}{ka_1^{k}} \sum_{\substack{\ell_1,\ell_2, \ell_3 ... \geq 0 \\ \ell_1 + 2\ell_2 + 3\ell_3 + \cdots = k-1}} (-1)^{\ell_1 + \ell_2 +\ell_3+\cdots} \frac{k \cdots (k-1+\ell_1+\ell_2+\cdots)}{\ell_1! \ell_2! \ell_3! \cdots}\left( \frac{a_2}{a_1}\right)^{\ell_1} \left( \frac{a_3}{a_1}\right)^{\ell_2} \cdots.
	\end{equation*}
\end{prop}

To deal with certain zeros of holomorphic functions, we require the following result from complex analysis, the proof of which is quickly obtained from Exercise 7.29 (i) in \cite{Burckel}. 

\begin{prop}\label{paper:T:CBiholMain}
	Let $r>0$ and let $\phi_n:B_r(0)\to\C$ be a sequence of holomorphic functions that converges uniformly on compact sets to a holomorphic function $\phi:B_r(0)\to\C$, with $\phi'(0)\ne0$. Then there exist $r>\k_1>0$ and $\k_2>0$ such that, for all $n$ sufficiently large, the restrictions $\phi_n|_{B_{\k_1}(0)}:B_{\k_1}(0)\to\phi_n(B_{\k_1}(0))$ are biholomorphic and $B_{\k_2}(0)\subset\phi_n(B_{\k_1}(0))$. In particular, the restrictions $\phi_n^{-1}|_{B_{\k_2}(0)}:B_{\k_2}(0)\to\phi_n^{-1}(B_{\k_2}(0))$ are biholomorphic functions. 
\end{prop}
 
\subsection{Asymptotic expansions}

We require two classes of asymptotic expansions.

\begin{Defn}%\label{paper:defn:AsyDef}
	Let $R\in\R$. 
	\begin{enumerate}[leftmargin=*,label=\textnormal{(\arabic*)}]
		\item\label{paper:item:AsyDef1} Let $g:\R^+\to\C$ be a function. Then $g\in\CK(R)$ if there exist real numbers $\nu_{g,1}<\nu_{g,2}<\nu_{g,3}<\dots<\nu_{g,N}<R$ and complex numbers $a_{g,j}$ such that
		\begin{equation*}%\label{paper:rhoexpansion}
			g(x) = \sum_{j=1}^{N_g} \frac{a_{g,j}}{x^{\nu_{g,j}}} + O_R\left(x^{-R}\right),\qquad (x\to\infty).
		\end{equation*}
		
		\item\label{paper:item:AsyDef2} Let $\phi$ be holomorphic on the right half-plane. Then $\phi\in\CH(R)$ if there are real numbers $\nu_{\phi,1}<\nu_{\phi,2}<\nu_{\phi,3}<\dots<\nu_{\phi,N}<R$ and $a_{\phi,j}\in\C$ such that, for all $k\in\N_0$ and $0<\d<\frac\pi2$,
		\begin{align}\label{paper:fasy}
			\phi^{(k)}(z) = \sum_{j=1}^{N_\phi} (\nu_{\phi, j})_k a_{\phi,j} z^{\nu_{\phi, j}-k} + O_{\d, R, k}\left( |z|^{R-k} \right), \qquad (z \to 0, z \in \CC_\d).
		\end{align}
	\end{enumerate}
	If there is no risk of confusion, then we write $N$, $\nu_j$, and $a_j$ in the above. The $R$-dependence of the error only matters if $R$ varies, for instance, if we can choose it to be arbitrarily large.
\end{Defn}

Note that any sequence $g(n)$ with
\begin{equation}\label{paper:eq:sequence-asy}
	g(n) = \sum_{j=1}^N \frac{a_j}{n^{\nu_j}} + O_R\left(n^{-R}\right),\qquad (n\to\infty),
\end{equation}
can be extended to a function $g$ in $\CK(R)$. Conversely, each function in $\CK(R)$ can be restricted to a sequence $g(n)_{n\in\N}$ satisfying \eqref{paper:eq:sequence-asy}. In addition, we include functions in $\CK(R)$ that have asymptotic expansion as in \ref{paper:item:AsyDef1}, but are initially defined only on intervals $(r,\infty)$ for some large $r>0$. The reason for this is that it does not matter how the function is defined up to $r$, and therefore it can always be continued to $(0,\infty)$. If $g\in\CK(R)$ for all $R>0$, then we write
\begin{equation}\label{paper:fexpand}
	g(x) = \sum_{j\ge1} \frac{a_j}{x^{\nu_j}},\qquad (x\to\infty).
\end{equation}
We use the same abbreviation if $\phi\in\CH(R)$ for all $R>0$. In this case we write $g\in\CK(\infty)$ and $\phi\in\CH(\infty)$, respectively. In some situations, we write for $R\in\R\cup\{\infty\}$
\begin{align*}
	g(x) = \sum_{j=1}^{N} \frac{a_{g,j}}{x^{\nu_{g,j}}} + O_R\left( x^{-R} \right),
\end{align*}
where $R$ might depend on the choice of the function $g$. If $R=\infty$, then one may ignore the error $O_R(x^{-R})$ and use the notation \eqref{paper:fexpand} instead. We have the following useful lemmas, that can be obtained by a straightforward calculation.

\begin{lem}\label{paper:Prop:Algebras}
	Let $R_1, R_2 \in \R$, $\lambda \in \C$, $g \in \CK(R_1)$, and $h \in \CK(R_2)$. Then we have the following:
	\begin{enumerate}[leftmargin=*,label=\textnormal{(\arabic*)}]
		\item\label{paper:Prop:Algebras:1} We have $\l g\in\CK(R_1)$ and $g+h\in\CK(\min\{R_1,R_2\})$. The exponents $\nu_{g+h,j}$ run through
		\begin{align*}
		 (\{ \nu_{g,j} \colon 1 \leq j \leq N_g \} \cup \{ \nu_{h,j} \colon 1 \leq j \leq N_h \}) \cap (-\infty, \min\{R_1, R_2\}).
		\end{align*}
		
		\item\label{paper:Prop:Algebras:2} We have $gh\in\CK(\min\{R_1+\nu_{h,1},R_2+\nu_{g,1}\})$. The exponents $\nu_{gh,j}$ run through
		\begin{align*}
			(\{ \nu_{g,j} \colon 1 \leq j \leq N_g \} + \{ \nu_{h,j} \colon 1 \leq j \leq N_h \}) \cap (-\infty, \min\{ R_1 + \nu_{h,1}, R_2 + \nu_{g,1}\}).
		\end{align*}
	\end{enumerate}
\end{lem}

We next deal with compositions of asymptotic expansions with holomorphic functions. 
	
\begin{lem}\label{paper:Prop:Holoexpansion}
	Let $0<R\leq \infty$, $g\in\CK(R)$ with $\nu_{g,1}=0$ and $h$ holomorphic at $a_{g,1}$. Then $(h\circ g)(x)$ is defined for all $x>0$ sufficiently large, and we have $h\circ g\in\CK(R)$ with
	\[
		\{\nu_{h\circ g,j} : 1\le j\le N_{h\circ g}\} = \left(\sum_{j=1}^{N_g}\nu_{g,j}\N_0\right) \cap [0,R).
	\]
\end{lem}

We need a similar result for general asymptotic expansions. 
	 
\begin{lem}\label{paper:Prop:AsyComp}
	Let $0<R_1,R_2\le\infty$, $\phi\in\CH(R_1)$, $g\in\CK(R_2)$, and $R:=\min\{R_2-\nu_{g,1},\nu_{g,1}R_1\}$. Assume $\nu_{g,1}>0$ and $g(x)>0$ for $x$ sufficiently large. Then $\phi\circ g\in\CK(R)$, $a_{\phi\circ g,1}=a_{\phi,1}a_{g,1}^{\nu_{\phi,1}}$, and
	\begin{align*}
	 \{ \nu_{\phi\circ g,j} \colon 1 \leq j \leq N_{\phi\circ g}\} = \left( \nu_{g,1} \{ \nu_{\phi, 1}, ..., \nu_{\phi, N_\phi}\} + \sum_{j=2}^{N_{g}} (\nu_{g,j} - \nu_{g,1}) \N_0\right) \cap (-\infty, R).
	\end{align*}
\end{lem}

\subsection{Special functions}

The following theorem collects some facts about the Gamma function.

\begin{prop}[see \cite{AAR,Tenenbaum}]\label{paper:GammaCollect}
	Let $\gamma$ denote the Euler--Mascheroni constant.
	\begin{enumerate}[leftmargin=*,label=\rm{(\arabic*)}]
		\item\label{paper:GC:1} The gamma function $\GG$ is holomorphic on $\C\setminus(-\N_0)$ with simple poles in $-\N_0$. For $n\in\N_0$ we have $\Res_{s=-n}\GG(s)=\frac{(-1)^n}{n!}$.
		
		\item \label{paper:GC:4} For $s = \sigma + it \in \C$ with $\s \in I$ for a compact interval $I \subset [\frac12, \infty)$, we uniformly have
		\begin{equation*}%\label{paper:Gamma}
			\max\left\{1, |t|^{\s - \frac12}\right\} e^{- \frac{\pi |t|}{2} } \ll_I |\Gamma(s)| \ll_I \max\left\{1, |t|^{\s - \frac12}\right\}e^{- \frac{\pi|t|}{2} }.
		\end{equation*}
		The bound also holds for compact intervals $I \subset \IR$ if $|t| \ge 1$.
		
		\item \label{paper:GC:3} Near $s=0$, we have the Laurent series expansion $\GG(s) = \frac{1}{s} - \gamma + O(s)$.
		
		\item \label{paper:GC:5} For all $s \in \C \setminus \Z$, we have $\Gamma(s)\Gamma(1-s) = \frac{\pi}{\sin(\pi s)}$. 
	\end{enumerate}
\end{prop}

For $s,z \in \IC$ with $s \notin -\N$, the {\it generalized Binomial coefficient} is defined by 
\begin{align*}
	\binom{s}{z} := \frac{\Gamma(s+1)}{\Gamma(z+1)\Gamma(s-z+1)}.
\end{align*}

We require the following properties of the Riemann zeta function.

\begin{prop}[see \cite{Apostol,Brue,Tenenbaum}]\label{paper:ZetaFunc} 
	\ \begin{enumerate}[leftmargin=*,label=\rm{(\arabic*)}]
		\item\label{paper:ZF:1} The $\z$-function has a meromorphic continuation to $\IC$ with only a simple pole at $s=1$ with residue $1$. For $s \in \IC$ we have (as identity between meromorphic functions)
		\begin{equation*}\label{paper:zetafunc}
			\zeta(s) = 2^s \pi^{s-1} \sin\left( \frac{\pi s}{2}\right) \Gamma(1-s) \zeta(1-s).
		\end{equation*} 
		%and equivalently 
		%\begin{align}
			%\label{paper:zetafunc2}	\zeta(1-s) = \frac{2}{(2\pi)^s} \cos\left( \frac{\pi s}{2}\right) \Gamma(s) \zeta(s).
		%\end{align}
	
		%\item \label{paper:ZF:2} Near $s=1$, the Riemann zeta function has a Laurent series expansion {\bf KB: If $\g$ is needed we cannot call our $\g$'s $\g$.} $\zeta(s) = \frac{1}{s-1}+\gamma + O(s)$.
	
		\item\label{paper:ZF:3} For $I := [\sigma_0, \sigma_1]$ and $s = \sigma +it \in \C$, there exists $m_I\in\Z$, such that for $\s \in I$
		\[
			\z(s) \ll (1+|t|)^{m_I},\qquad (|t|\to\infty).
		\]
		
		\item\label{paper:ZF:4} Near $s=1$, we have the Laurent series expansion $\zeta(s) = \frac{1}{s-1}+\gamma + O(s-1)$.
	\end{enumerate}
\end{prop}

For the Saddle Point Method we need the following estimate.

\begin{lem}\label{paper:Gauss}
	Let $\mu_n$ be an increasing unbounded sequence of positive real numbers, $B>0$, and $P$ a polynomial of degree $m\in\N_0$. Then we have 
	\begin{align*}
		\int_{-\mu_n}^{\mu_n} P(x) e^{-Bx^2} dx = \int_{-\infty}^{\infty} P(x) e^{-Bx^2} dx + O_{B,P}\left( \mu_n^{\frac{m-1}{2}} e^{-B\mu_n^2} \right).
	\end{align*}
\end{lem}

Finally, we require the following in our study of the Witten zeta function $\zeta_{\so(5)}.$

\begin{lem}\label{paper:L:IntegralPolyBound}
	Let $n\in\N_0$. The function $g:\R\to\R$ defined as $g(u):=e^{|u|}\int_{-\infty}^\infty|v|^ne^{-|v|-|v+u|}dv$ satisfies $g(u)=O_n(u^{n+1})$ as $|u|\to\infty$.
\end{lem}

\begin{proof}
	Let $u\ge0$. Then we have
	\[
		g(u) = \frac{n!}{2^{n+1}}\sum_{j=0}^n \frac{2^j}{j!}u^j + \frac{u^{n+1}}{n+1} + \frac{n!}{2^{n+1}} = O_n\left(u^{n+1}\right).
	\]
	The lemma follows, since $g$ is an even function.
\end{proof}

\section{Minor and major arcs}\label{sec:minmajarc}

\subsection{The minor arcs}\label{paper:sec:minorarcs}

For $z \in \C$ with $\operatorname{Re}(z) > 0$, we define, with $G_f$ given in \eqref{def:GfLf},
$$
	\Phi_f(z):=\Log(G_f(z)).
$$ 
Note that we assume throughout, that the function $f$ grows polynomially, which is implicitly part of \ref{paper:main:3}. We apply Cauchy's Theorem, writing
\begin{align*}
 p_f(n) = \frac{1}{2\pi} \int_{-\pi}^{\pi} \exp\left( n(\varrho_n + it) + \Phi_f(\varrho_n + it)\right)dt,
\end{align*}
where {$\varrho_n \to 0^+$} is determined in Subsection \ref{paper:subsect:saddlepoint}. We split the integral into two parts, the major and minor arcs, for any $\b\ge1$
\begin{align} \label{paper:MajorMinorSplit}
	p_f(n) = \frac{e^{\varrho_n n}}{2\pi}\int_{|t| \leq \varrho_n^\b} \exp\left( int + \Phi_f(\varrho_n + it)\right)dt + \frac{e^{\varrho_n n}}{2\pi} \int_{\varrho^\b \leq |t| \leq \pi} \exp\left( int + \Phi_f(\varrho_n + it)\right)dt.
\end{align}
The first integral provides the main terms in the asymptotic expansion for $p_f(n)$, the second integral is negligible, as the following lemma shows.

\begin{lem}\label{paper:L:generalminorarcsbound} 
	Let $1<\b<1+\frac\a2$ and assume that $f$ satisfies the conditions of Theorem \ref{paper:T:main2}. Then
	\begin{align*}
 	\int_{\frac{\varrho_n^\beta}{2\pi} \leq |t| \leq \frac{1}{2}} e^{2\pi int} G_f(\varrho_n + 2 \pi i t) dt \ {\ll_L \varrho_n^{L+1} G_f(\varrho_n)}.
	\end{align*}
\end{lem}

\begin{proof}[Sketch of proof.]
	The proof may be adapted from \cite[Lemma 3.1]{DebrTen}. That is, we estimate the quotient,
	\begin{equation*}%\label{E:Quotientinitialboundpaper}
		\frac{|G_f(\p_n+2\pi it)|}{G_f(\p_n)} \leq \prod_{m\geq1} \left(1+\frac{16\nor{mt}^2}{e^{m\p_n}m^2\p_n^2}\right)^{-\frac{f(m)}{2}},
	\end{equation*}
	where $||x||$ is the distance from $x$ to the nearest integer. We then throw away $m$-th factors depending on the location of $t\in[\frac{\p_n^\b}{2\pi},\frac12]$. The proof follows \cite[Lemma 3.1]{DebrTen} {\it mutatis mutandis}; key facts are hypothesis (P3) of Theorem \ref{paper:T:main2} and that (which follows from \cite[Theorem 7.28 (1)]{Tenenbaum})
\begin{equation*}
			\sum_{1 \leq m \leq x} f(m) \sim \frac{\mathrm{Res}_{s=\alpha} L_f}{\alpha} x^{\alpha}. \qedhere
\end{equation*}	
\end{proof}

\subsection{Inverse Mellin transforms for generating functions}

We start this subsection with a lemma on the asymptotic behavior of the function $\Phi_f$ near $z=0$.

\begin{lem}\label{paper:C:Phi-Asymp}
	Let $f:\N\to\N_0$ satisfy \ref{paper:main:3} with $R>0$ and \ref{paper:main:4}. Fix some $0<\d<\frac\pi2-a$. Then we have, as $z\to0$ in $C_\d$,
	\[
		\Phi_f(z) = \sum_{\substack{\nu\in -\mathcal{P}_R\setminus\{0\}}} \Res_{s=-\nu} L_f^*(s)z^{\nu} - L_f(0)\Log(z) + L_f'(0) + O_{R}\left(|z|^R\right).
	\]
	For the $k$-th derivative ($k\in\N$), we have
	\[
		\Phi_f^{(k)}(z) = \sum_{\substack{\nu\in -\mathcal{P}_R\setminus\{0\}}} (\nu)_k \Res_{s=-\nu} L_f^*(s)z^{\nu - k} + \frac{(-1)^k(k-1)! L_f(0)}{z^k} + O_{R,k}\left(|z|^{R-k}\right).
	\]
\end{lem}

\begin{proof}
	With $J_f(s;z):=L_f^*(s)z^{-s}$, we obtain, {for $\k\in\N_0$},
	\begin{align}\label{paper:eq:Phikderiv}
		2\pi i\Phi_f^{(\kappa)}(z) = \pdd z\kappa \left(\hspace{.1cm}\int_{-R-i\infty}^{-R+i\infty} + \lim_{K\to\infty} \left(\int_{\del\CR_{-R,\a+1;K}} + \int_{\a+1-iK}^{-R-iK} + \int_{-R+iK}^{\a+1+iK}\hspace{.1cm}\right)\right) J_f(s;z) ds.
	\end{align}
	Here we use \ref{paper:main:3}, giving that there are no poles of $J_f(s;z)$ on the path of integration. By Proposition \ref{paper:ZetaFunc} \ref{paper:ZF:3}, \cite[Theorem. 2.1 (3)]{BF}, and \ref{paper:main:4}, we find a constant $c(R,\k)$ such that, as $|v|\to\infty$, 
	\begin{align*}
		\left|L_f^*(-R+iv)\right| \ll_{R} (1+|v|)^{c(R,\kappa)} e^{-\left(\frac{\pi}{2}-a\right)|v|}.
	\end{align*}
	This yields, with Leibniz' integral rule and $0<\d<\frac\pi2-a$,
	\begin{align*}
		\left|\pdd{z}{\kappa}\int_{-R-i\infty}^{-R+i\infty} J_f(s;z) ds\right|%%		\ll_{R} |z|^{R-k}\int_{-\infty}^\infty (1+|v|)^{c(R,k)}e^{-\d|v|} dv
		\ll_{R,\kappa} |z|^{R-\kappa}. 
	\end{align*}
	
	For the second integral in \eqref{paper:eq:Phikderiv}, applying the Residue Theorem gives
	\begin{multline*}
		\pdd z\kappa \lim_{K\to\infty}\frac{1}{2\pi i}\int_{\del\CR_{-R,\a+1;K}} J_f(s;z) ds\\
		= \sum_{\substack{\nu\in -\mathcal{P}_R\setminus\{0\}}} (\nu)_\kappa\Res_{s=-\nu} L_f^*(s) z^{\nu-\kappa} + \pdd z\kappa \left(-L_f(0)\Log(z)+L_f'(0)\right),
	\end{multline*}
	since $s=0$ is a double pole of $J_f(s;z)$. For the last two integrals in \eqref{paper:eq:Phikderiv} we have, for some ${m(I)}\in\N_0$, depending on $I:={[-R,\a+1]}$, 
	\[
		\left|\int_{-R\pm iK}^{\a+1\pm iK} J_f(s;z) ds\right| \ll_I (1+|K|)^{m(I)}\max\left\{|z|^{\a+1},|z|^{-R}\right\}e^{-(\d-a)|K|},
	\]
	which vanishes as $K\to\infty$ and thus the claim follows {by distinguishing $\kappa=0$ and $\kappa\in\N$.} 
\end{proof}

\subsection{Approximation of saddle points}\label{paper:subsect:saddlepoint}

We now approximately solve the saddle point equations
\begin{align} \label{paper:eq:saddle-point}
	-\Phi'_f(\varrho) = n= - \Phi'_f(\varrho_n).
\end{align}
The following proposition provides an asymptotic formula for certain functions.

\begin{prop}\label{paper:T:maininverse}
	Let $\phi\in\CH(R)$ with $R>0$, $\nu_{\phi,1}<0$, and $a_{\phi,1}>0$. Assume that $\phi(\R^+)\subset\R$. %$\a > 0$ and $\phi$ a function holomorphic on the right-half plane, satisfying $\phi|_{\R^+} \subset \R$. Assume that there exist $R > 0$ and $N \in \N$, such that for all $k \in \N_0$, 
 	%\begin{align} \label{paper:fasy}
 	% \phi^{(k)}(z) = \sum_{\lambda\in\CI} (\lambda)_k c_\lambda z^{\lambda-k} + O\left(z^{R-k}\right), \qquad (z \to 0, \, z \in \mathcal{C}_\delta)
 	%\end{align}
 	%where $0 < \delta < \frac{\pi}{2}$, $c_\lambda \in \R$, $c_{-\a-1} > 0$, and $-a-1\in\CI \subset [-\a-1,R]$ is a set of $N$ elements. %$-\a-1 = \lambda_1 < \lambda_2 < \cdots < \lambda_N < R$.
 	Then we have the following:
 	\begin{enumerate}[leftmargin=*,label=\textnormal{(\arabic*)}]
 		\item\label{paper:item:maininverse:1} There exists a positive sequence $(\p_n)_{n\in\N}$, such that for all $n$ sufficiently large, $\phi(\p_n)=n$ holds.
 		
 		\item\label{paper:item:maininverse:2} We have\footnote{Recall that we can consider the sequence $\p_n$ as a function on $\R^+$.} $\p\in\CK(1-\frac{R+1}{\nu_{\phi,1}})$, $a_{\p,1}=a_{\phi,1}^{-\frac{1}{\nu_{\phi,1}}}$, %where the expansion
 		%\begin{equation}\label{paper:rhoexpansion}
 			%$$	\varrho_n = \left(\frac{a_{\phi,\nu_{\phi,1}}}{n}\right)^{-\frac{1}{\nu_{\phi,1}}} + \sum_{\b\in\CJ} \frac{u_\b }{n^{\beta}} + O\left( n^{\frac{R+1}{\nu_{\phi,1}}-1}\right), \qquad n \to \infty,$$
 		%\end{equation}
 		and the corresponding exponent set %the $\frac{1}{\a+1} = \b_1 < \b_2 < \cdots < \b_L$ are the elements of
 		\begin{equation*}%\label{paper:eq:exponentsetrhofin}
 			\{\nu_{\p,j} : 1\le j\le N_\p\} = \left(-\frac{1}{\nu_{\phi,1}}+\sum_{j=1}^{N_\phi} \left(1-\frac{\nu_{\phi,j}}{\nu_{\phi,1}}\right)\N_0\right)\cap \left(-\infty,1-\frac{R+1}{\nu_{\phi,1}}\right).
 		\end{equation*}
 		In particular, we have $\p_n \to 0^+$. 
	\end{enumerate}
\end{prop}

\begin{proof}
	In the proof we abbreviate $\nu_n:=\nu_{\phi,n}$ and $a_n:=a_{\phi,n}$. \\
	(1) For $n\in\N$, set
 	\[
		\psi_n(w) := -1 + \frac1n\phi\left(\left(\frac{n}{a_1}\right)^\frac{1}{\nu_1}w\right).
	\]
	As $\phi$ is holomorphic on the right-half plane by assumption, so are the $\psi_n$. Using \eqref{paper:fasy}, write
	\begin{equation}\label{paper:eq:gnexpand}
		\psi_n(w) = w^{\nu_1} -1 + E_n(w),
	\end{equation}
	where the error satisfies
	\[
		E_n(w) = \frac1n\sum_{\substack{j=2}}^{N_\phi} a_j\left(\frac{n}{a_1}\right)^\frac{\nu_j}{\nu_1}w^{\nu_j} + O_R\left(n^{\frac{R}{\nu_1}-1}|w|^R\right).
	\]
	We next show that, for all $n$ sufficiently large, the $\psi_n$ only have one zero near $w=1$. We argue with Rouch\'e's Theorem. First, we find that, for $n$ sufficiently large, the inequality
	\begin{align} \label{paper:eq:EnIneq}
		|E_n(w)| < \left| 1 - w^{\nu_1}\right| + \left| w^{\nu_1} - 1 + E_n(w) \right| = \left| 1 - w^{\nu_1}\right| + |\psi_n(w)|
	\end{align}
	holds on the entire boundary of $B_{\k(\nu_1)}(1)$, with $0<\k(\nu_1)<\frac12$ sufficiently small such that $w\mapsto1-w^{\nu_1}$ only has one zero in $B_{\k(\nu_1)}(1)$. By Rouch\'e's Theorem and \eqref{paper:eq:EnIneq}, for $n$ sufficiently large $\psi_n$ also has exactly one zero in $B_{\k(\nu_1)}(1)$. We denote this zero of $\psi_n$ by $w_n$. It is real as $\phi$ is real-valued on the positive real line and a holomorphic function. One can show that $\p_n=(\frac{n}{a_1})^\frac{1}{\nu_1}w_n>0$ satisfies $\phi(\p_n)=n$.\\
	(2) We first give an expansion for $w_n$. By Proposition \ref{paper:T:CBiholMain}, there exists $\k>0$, such that for all $n$ sufficiently large and all $z\in B_\k(0)$, the inverse functions $\psi_n^{-1}$ of $\psi_n$ are defined and holomorphic in $B_\k(1)$. Using this, we can calculate $w_n$, satisfying $\psi_n(w_n)=0$. For this, let
	\[
		h_n(w):=\psi_n(w+1)-\psi_n(1).
	\]
	We have $h_n(0)=0$, and we find, with \Cref{paper:T:InverseFormula},
	\[
		w_n - 1 = h_n^{-1}(-\psi_n(1)) = \sum_{m \ge1} (-1)^mb_m(n)\psi_n(1)^m,
	\]
	where the $b_m$ can be explicitly calculated. First, $\psi_n(1)^m$ ($m\in\N_0$) have expansions in $n$ by \eqref{paper:eq:gnexpand} and \Cref{paper:Prop:Holoexpansion}. They have exponent set $\sum_{2\le j\le N_\phi}(1-\frac{\nu_j}{\nu_1})\N_0\cap[0,1-\frac{R}{\nu_1})$. We find, for $k\in\N$,
	\begin{equation}\label{paper:eq:gnDeriv}
		\psi_n^{(k)}(1) = \frac1n\sum_{j=1}^{N_\phi} (\nu_j)_ka_j\left(\frac{n}{a_1}\right)^\frac{\nu_j}{\nu_1} + O_R\left(n^{\frac{R}{\nu_1}-1}\right).
	\end{equation}
	Again by \Cref{paper:Prop:Holoexpansion}, and \eqref{paper:eq:gnDeriv}, $\psi^{(k)}_n(1)$ ($k\in\N_0$) has expansions in $n$, with exponent set $(\sum_{2\le j\le N_\phi}(1-\frac{\nu_j}{\nu_1})\N_0)\cap[0,1-\frac{R}{\nu_1})$. By \Cref{paper:Prop:Holoexpansion} we have the following expansion in $n$
	\begin{align*}
		\psi_n'(1)^{-m} = \left(\nu_1+\frac1n\sum_{j=2}^{N_\phi} \nu_ja_j\left(\frac{n}{a_1}\right)^\frac{\nu_j}{\nu_1} + O_R\left(n^{\frac{R}{\nu_1}-1}\right)\right)^{-m}
	\end{align*}
	with exponent set $(\sum_{2\le j\le N_\phi}(1-\frac{\nu_j}{\nu_1})\N_0)\cap[0,1-\frac{R}{\nu_1})$. By the formula in \Cref{paper:T:InverseFormula}, the $b_m(n)$ are essentially sums and products of terms $\psi_n'(1)^{-1}$ and $\psi^{(k)}_n(1)$, where $k\ge2$. Hence, $b_m(n)$ has an expansion in $n$, with exponent set $(\sum_{2\le j\le N_\phi}(1-\frac{\nu_j}{\nu_1})\N_0)\cap[0,1-\frac{R}{\nu_1})$, and according to \Cref{paper:Prop:Algebras}, the same holds for finite linear combinations $\sum_{1\le m\le M}(-1)^mb_m(n)\psi_n(1)^m$. As $\psi_n(1)=O(n^{\frac{\nu_2}{\nu_1}-1})$ for $n\to\infty$, one has, for $M$ sufficiently large and not depending on $n$,
	\[
		\sum_{m\geq M+1} (-1)^mb_m(n)\psi_n(1)^m = O_R\left(n^{\frac{R}{\nu_1}-1}\right).
	\]
	Now, as $w_n\sim1$, we conclude the theorem recalling that $\p_n=(\frac{n}{a_1})^\frac{1}{\nu_1}w_n$.
\end{proof}

We next apply \Cref{paper:T:maininverse} to $-\Phi_f'$. For the proof one may use \Cref{paper:C:Phi-Asymp} with $k=1$.

\begin{cor}\label{paper:cor:rhosaddleexp1}
	Let $\p_n$ solve \eqref{paper:eq:saddle-point}. Assume that $f\colon\N\to\N_0$ satisfies the conditions of Theorem \ref{paper:T:main2}. Then $\p\in\CK(\frac{R}{\a+1}+1)$ with $a_{\p,1}=a_{-\Phi_f',1}^\frac{1}{\a+1} = (\w_\a\GG(\a+1)\z(\a+1))^{\frac{1}{\a+1}}$ and we have 
	\begin{align*}
		\{\nu_{\p,j}\colon 1\leq j\leq N_\p\} =\left(\frac{1}{\a+1} - \sum_{\substack{\mu\in \textcolor{black}{\mathcal{P}_R}}}\left( \frac{\mu+1}{\a+1} - 1 \right) \N_0\right)\cap \left[\frac{1}{\a + 1},\frac{R}{\a + 1} + 1\right).
	\end{align*}
\end{cor}

\subsection{The major arcs}\label{paper:sec:majorarcs}

In this subsection we approximate, for some $1+\frac\a3<\b<1+\frac\a2$, 
\[
	I_n := \int_{|t|\le\p_n^\b} \exp(\Phi_f(\p_n+it)+int) dt,
\]
where $\a$ is the largest positive pole of $L_f$. The following lemma can be shown using \cite[\S 4]{DebrTen}.

\begin{lem}\label{paper:L:InCalc2}
	Let $f:\N\to\N_0$ satisfy the conditions of Theorem \ref{paper:T:main2}, $\p_n$ solve \eqref{paper:eq:saddle-point}, and $N\in\N$. Then we have
	\begin{align*} %\label{paper:eq:Integralexpansion}
		I_n = \sqrt{2\pi} G_f(\varrho_n) \left( \frac1{\sqrt{\Phi''_f(\varrho_n)}} + \sum_{2 \leq k \leq \frac{3H(N+\a)}{2\a}} \frac{(2k)! \lambda_{2k}(\varrho_n)}{2^k k! \Phi''_f(\varrho_n)^{k+\frac12}} + O_N\left( \varrho_n^N \right) \right),
	\end{align*}
	where $H:=\ceil{\frac{N}{3(\b-1-\frac\a3)}}+1$ and
	\begin{align*}
 	\lambda_{2k}(\varrho) := (-1)^k \sum_{h=1}^{H} \frac{1}{h!} \sum_{\substack{3 \leq m_1, ..., m_h \leq \frac{3(N+\a)}{\a} \\ m_1 + \cdots + m_h = 2k}} \prod_{j=1}^h \frac{\Phi_f^{(m_j)}(\varrho)}{m_j!}.
	\end{align*}
\end{lem}

The following lemma shows that the first term in \Cref{paper:L:InCalc2} dominates the others; its proof follows with \Cref{paper:Prop:AsyComp}, \Cref{paper:C:Phi-Asymp}, and \Cref{paper:cor:rhosaddleexp1} by a straightforward calculation.

\begin{lem}\label{paper:L:lambda}
	Let $k\geq 2$ and assume the conditions as in Lemma \ref{paper:L:InCalc2}. Then we have %$\frac{\lambda_{2k}(\varrho_n)}{\Phi''_f(\varrho_n)^{k+\frac12}} = o( {\Phi''_f(\varrho_n)^{-\frac12}})$.
	\[
		\frac{\l_{2k}(\p_n)}{\Phi''_f(\p_n)^{k+\frac12}} = \sum_{j=1}^M \frac{b_j}{n^{\eta_j}} + O_R\left(n^{-{R+1+\left(k-\pflo{\frac{2k}{3}}+\frac32\right)}\frac{\a}{\a+1}}\right), 
	\]
	where the $\eta_j$ run through
 \[
 	\frac{\a+2}{2(\a+1)} + \frac{\a}{\a+1}\N_0 + \left(-\sum_{\mu\in\CP_R} \left(\frac{\mu+1}{\a+1}-1\right)\N_0\right) \cap \left[0,\frac{R+\a}{\a+1}\right).
 \]
\end{lem}

We next use Lemma \ref{paper:Prop:AsyComp} and Corollary \ref{paper:cor:rhosaddleexp1} to give an asymptotic expansion for $G_f(\p_n)$.

\begin{lem}\label{paper:L:Gfasy}
	Assume that $f:\N\to\N_0$ satisfies the conditions of Theorem \ref{paper:T:main2}. Then, we have 
	\begin{align*}
 	G_f(\varrho_n) & = \frac{e^{L_f'(0)}n^{\frac{L_f(0)}{\a+1}}}{a_{-\Phi_f',1}^{\frac{L_f(0)}{\a+1}}} \exp\left( \frac{1}{\a}(\omega_\a \GG(\a+1)\z(\a+1))^{\frac{1}{\a+1}} n^{\frac{\a}{\a+1}} + \sum_{j=2}^M C_j n^{\b_j}\right) \\
 	& \hspace{4cm} \times \left( 1 + \sum_{j = 1}^N \frac{B_j}{n^{\d_j}} + O_R\left( n^{-\frac{R}{\a+1}}\right)\right),
	\end{align*}
	where $0\le\b_M<\dots<\b_2<\frac{\a}{\a+1}$ run through $\CL$ and $0<\d_1<\d_2<\dots<\d_N$ through $\CM+\CN$.
 %\textcolor{orange}{\begin{align*}
 % \left(- \sum\limits_{\nu\in\mathcal{P}_R} \left( \tfrac{\nu+1}{\a+1} - 1\right) \N_0\right)\cap \left[0,\tfrac{R-1}{\a+1} + 1\right)
 %\end{align*}}
\end{lem}

\begin{proof}
	Let $\phi(z):=\Phi_f(z)+L_f(0)\Log(z)$ and $F:=\phi\circ\p$. By Lemma \ref{paper:C:Phi-Asymp}, Proposition \ref{paper:T:maininverse}, and Lemma \ref{paper:Prop:AsyComp} we find that
	\begin{align} \label{paper:eq:PhiRho}
		\Phi_f(\varrho_n) + L_f(0)\log(\varrho_n) = L'_f(0) + \sum_{j = 1}^{N_F} \frac{a_{F,j}}{ n^{\nu_{F,j}}} + O_R\left(n^{-\frac{R}{\a+1}}\right),
	\end{align}
	where $\nu_{F,j}$ run through (the inclusion follows by Corollary \ref{paper:cor:rhosaddleexp1})
	\begin{align}\nonumber
		& \left(-\frac{1}{\a+1}\CP_R + \sum_{j=2}^{N_\p} \left(\nu_{\p,j}-\frac{1}{\a+1}\right)\N_0\right) \cap \left(-\infty, \frac{R}{\a+1}\right) \\
		\label{paper:eq:set5}
		& \hspace{3cm} \subset \left( -\frac{1}{\a+1}\mathcal{P}_R - \sum\limits_{\mu\in\mathcal{P}_R}\left(\frac{\mu+1}{\a+1} - 1\right) \N_0\right) \cap \left(-\infty, \frac{R}{\a+1}\right).
	\end{align}
	%, we find that the set \eqref{paper:eq:set5} is contained in
	%\begin{align*}
		%& -\tfrac{1}{\a+1}\mathcal{P} - \sum\limits_{\nu\in\mathcal{P}}\left(\tfrac{\nu+1}{\a+1} - 1\right) \N_0.
	%\end{align*}
	Note that, again by \Cref{paper:Prop:AsyComp} and \Cref{paper:C:Phi-Asymp}, we obtain
	\begin{align*}
		a_{F,1} = a_{\phi,1} a_{\p,1}^{\nu_{\phi,1}} = \frac{1}{\a} (\omega_\a \GG(\a+1) \z(\a+1))^{\frac{1}{\a+1}}.
	\end{align*}
	We split the sum in \eqref{paper:eq:PhiRho} into two parts: one with nonpositive $\nu_{F,1},\dots,\nu_{F,M}$, say, and the one with positive $\nu_{F,j}<\frac{R}{\a+1}$. Note that $M$ is bounded and independent of $R$. Exponentiating \eqref{paper:eq:PhiRho} yields
	\begin{equation*}%\label{paper:eq:PhiRhoExp}
		\exp(\Phi_f(\p_n)) = \p_n^{-L_f(0)}e^{L_f'(0)}\exp\left(\sum_{j=M+1}^{N_{F}} \frac{a_{F,j}}{n^{\nu_{F,j}}} + O_R\left( n^{-\frac{R}{\a+1}}\right)\right)\exp\left(\sum_{j=1}^M \frac{a_{F,j}}{n^{\nu_{F,j}}}\right).
	\end{equation*}
	Note that the positive $\nu_{F,j}$ run through \eqref{paper:eq:set5} with $-\infty$ replaced by $0$. By \Cref{paper:Prop:Holoexpansion}, we have
	\[
		\exp\left(\sum_{j=M+1}^{N_F} \frac{a_{F,j}}{n^{\nu_{F,j}}} + O_R\left(n^{-\frac{R}{\a+1}}\right)\right) = 1 + \sum_{j=1}^{K} \frac{H_j}{n^{\e_j}} + O_R\left(n^{-\frac{R}{\a+1}}\right)
	\]
	for some $K\in\N$ and with exponents $\e_j$ running through $\CN$. Recall that, by \Cref{paper:cor:rhosaddleexp1}, we have $\p_n\sim a_{\p,1}n^{-\frac{1}{\a+1}}$. Now set $h(n):=n^{-\frac{L_f(0)}{\a+1}}\p_n^{-L_f(0)}$. A straightforward calculation using \Cref{paper:cor:rhosaddleexp1} shows that $h\in\CK(\frac{R+\a}{\a+1})$ with exponent set $(-\sum_{\mu\in\CP_R}(\frac{\mu+1}{\a+1}-1)\N_0)\cap[0,\frac{R+\a}{\a+1})\subset\CM$ and $a_{h,1}=a_{-\Phi_f',1}^{-\frac{L_f(0)}{\a+1}}$. By \Cref{paper:Prop:Algebras} \ref{paper:Prop:Algebras:2}, we obtain, for some $N\in\N$, $B_j\in\C$, and $\d_j$ running through $\CM+\CN$,
	\[
		h(n)\left(1+\sum_{j=1}^K \frac{H_j}{n^{\e_j}}+O_R\left(n^{-\frac{R}{\a+1}}\right)\right) = a_{h,1}\left(1+\sum_{j=1}^N \frac{B_j}{n^{\d_j}}+O_R\left(n^{-\frac{R}{\a+1}}\right)\right).
	\]
	Setting $C_j:=a_{F,j}$ for $1\le j\le M$, the lemma follows.
\end{proof}

Another important step for the proof of our main theorem is the following lemma. % deduced from Lemma \ref{paper:Prop:Holoexpansion} and Corollary \ref{paper:cor:rhosaddleexp1}.

\begin{lem}\label{paper:lem:expnrho}
	Let $f:\N\to\N_0$ satisfy the conditions of \textcolor{black}{Theorem \ref{paper:T:main2}}. Then we have, as $n\to\infty$,
 \begin{align*}
 e^{n\p_n} = \exp\left((\omega_\a \Gamma(\a+1)\zeta(\a+1))^\frac{1}{\a+1}n^{\frac{\a}{\a+1}}+\sum\limits_{j=2}^{M} a_{\p,j}n^{\eta_{j}} \right) \left(1 + \sum_{j=1}^N \frac{D_j}{n^{\mu_j}} + O_R\left( n^{-\frac{R}{\a+1}}\right)\right)
 \end{align*}
 for some $1\le M\le N_\p$, with $\frac{\a}{\a+1}>\eta_2>\dots>\eta_M\ge0$ running through $\CL$ and the $\mu_j$ through $\CN$.
\end{lem}

\begin{proof}
	Let $g(n):=n\p_n$. By \Cref{paper:cor:rhosaddleexp1} we have $g\in\CK(\frac{R}{\a+1})$ with exponent set
	\[
		\{\nu_{g,j} : 1\le j\le N_\p\} = \left(-1+\frac{1}{\a+1}-\sum_{\substack{\mu\in\CP_R}} \left(\frac{\mu+1}{\a+1}-1\right)\N_0\right) \cap \left[-1+\frac{1}{\a+1},\frac{R}{\a+1}\right).
	\]
	Hence, for some $1\leq M\leq N_\p$, we obtain 
	\begin{align*}
 e^{n\p_n} = \exp\left(a_{-\Phi_f',1}^\frac{1}{\a+1}n^{\frac{\a}{\a+1}}+\sum\limits_{j=2}^{M} \frac{a_{\p,j}}{n^{\nu_{g,j}}} \right) \exp\left(\sum\limits_{j=M+1}^{N_\p} \frac{a_{\p,j}}{n^{\nu_{g,j}}} + O_R\left(n^{-\frac{R}{\a+1}}\right)\right)
 \end{align*}
	with $-\frac{\a}{\a+1}<\nu_{g,2}<\dots<\nu_{g,M}\le0<\nu_{g,M+1}<\dots<\nu_{g,N_\p}$. By Lemma \ref{paper:C:Phi-Asymp} we obtain $a_{-\Phi_f',1}^\frac{1}{\a+1}=(\w_\a\GG(\a+1)\z(\a+1))^\frac{1}{\a+1}$. %The $\tfrac{\a}{\a+1} > - \nu_2 > \dots > - \nu_M \geq 0$ are contained in
	%\begin{align*}
	%	\tfrac{\a}{\a+1} + \sum_{\substack{\nu\in \textcolor{black}{\mathcal{P}_R}}}\left( \tfrac{\nu+1}{\a+1} - 1 \right) \N_0,
	%\end{align*}
	%and this is a subset of $\mathcal{L}$, as $\a \in \CP_R$. 
	Note that the exponents $0<\nu_{g,M+1}<\dots<\nu_{g,N_\p}$ run through
	\[
		\left(-\frac{\a}{\a+1}-\sum_{\substack{\mu\in\CP_R}} \left(\frac{\mu+1}{\a+1}-1\right)\N_0\right) \cap \left(0,\frac{R}{\a+1}\right).
	\]
	By Lemma \ref{paper:Prop:Holoexpansion}, $\exp(\sum_{j=M+1}^{N_\p}\frac{a_{\p,j}}{n^{\nu_{g,j}}}+O_R(n^{-\frac{R}{\a+1}}))$ is in $\CK(\frac{R}{\a+1})$, with exponent set
	\[
		\left\{\sum_{j=1}^K b_j\th_j : K,b_j\in\N_0,\ \th_j\in\left(-\frac{\a}{\a+1}-\sum_{\substack{\mu\in\CP_R}} \left(\frac{\mu+1}{\a+1}-1\right)\N_0\right)\cap\left(0,\frac{R}{\a+1}\right)\right\}. %= \CN.\tag*{\qedhere}
	\]
	As $\a\in\CP_R$, this is a subset of $\CN$, so the above exponents are given by $\CN$, proving the lemma.
\end{proof}

The following corollary is very helpful to prove our main theorem.

\begin{cor}\label{paper:C:expTimesGf}
	Let $f \colon \N \to \N_0$ satisfy the conditions of Theorem \ref{paper:T:main2}. Then we have
	\begin{align*}
		e^{n\p_n} G_f(\p_n) = \frac{e^{L_f'(0)}n^{\frac{L_f(0)}{\a+1}}}{a_{-\Phi_f',1}^{\frac{L_f(0)}{\a+1}}} \exp\left( A_1 n^{\frac{\a}{\a+1}} + \sum_{j=2}^M A_j n^{\a_j}\right)\left( 1 + \sum_{j = 1}^N \frac{E_j}{n^{\eta_j}} + O_R\left( n^{-\frac{R}{\a+1}}\right)\right),
	\end{align*}
	with $A_1$ defined in \eqref{paper:eq:mainconstants2}, $\frac{\a}{\a+1}>\a_2>\dots>\a_M\ge0$ running through $\CL$, and $\eta_j$ through $\CM+\CN$.
\end{cor}

\section{Proof of Theorem \ref{paper:T:main2}}\label{sec:proof1.4}

\subsection{The general case}

The following lemma follows by a straightforward calculation, using \eqref{paper:MajorMinorSplit} and Lemmas \ref{paper:L:InCalc2}, \ref{paper:L:generalminorarcsbound}, and \ref{paper:L:lambda}.

\begin{lem}\label{paper:L:mainasy}
	Let $f:\N\to\N_0$ satisfy the conditions of Theorem \ref{paper:T:main2}. Then we have
	\begin{align*}
		p_f(n) = \frac{e^{n\p_n}G_f(\p_n)}{\sqrt{2\pi}}\left(\sum\limits_{j=1}^M \frac{d_j }{n^{\nu_j}} + O_{L,R}\left(n^{-\min\left\{\frac{L+1}{\a+1}, \frac{R+\a}{\a+1}+\frac{\a+2}{2(\a+1)}\right\}}\right)\right)
	\end{align*}
	for some $M\in\N$, $d_1=\frac{1}{\sqrt{\a+1}}(\w_\a\GG(\a+1)\z(\a+1))^\frac{1}{2(\a+1)}$, and the $\nu_j$ run through
	\[
		\frac{\a+2}{2(\a+1)} + \frac{\a}{\a+1}\N_0 + \left(-\sum_{\mu\in\CP_R} \left(\frac{\mu+1}{\a+1}-1\right)\N_0\right) \cap \left[0,\frac{R+\a}{\a+1}\right).
	\]
	In particular, we have $\nu_1 =\tfrac{\a+2}{2(\a+1)}$.
\end{lem}

We prove the following lemma.

\begin{lem}\label{paper:lem:One-Pole-Saddle}
	Assume that $f$ satisfies the conditions of Theorem \ref{paper:T:main2} and that $L_f$ has only one positive pole $\a$. Then we have 
\begin{align*}
 n \varrho_n + \Phi_f(\varrho_n) = \left(\omega_{\a} \Gamma(\a+1)\z(\a+1)\right)^{\frac{1}{\a+1}} \left(1 + \tfrac{1}{\a}\right)n^{\frac{\a}{\a+1}} - L_f(0) \log(\varrho_n) + L'_f(0) + o(1). 
\end{align*}
\end{lem}

\begin{proof}
	By Lemma \ref{paper:C:Phi-Asymp}, we have
	\begin{align} \label{paper:eq:Phi-simple}
		\Phi_f(\varrho_n) = \frac{\omega_{\a} \Gamma(\a)\z(\a+1)}{\varrho_n^\a} - L_f(0)\log(\varrho_n) + L'_f(0) + O\left(\varrho_n^{R_0}\right),
	\end{align}
	where
	\begin{align}\label{paper:eq:R0defi}
		R_0 :=
		\begin{cases}
			\us{\nu\in\CP_R\cap(-R,0)}{-\max\nu} & \text{if }\CP_R\cap(-R,0)\ne\emptyset,\\
			R & \text{otherwise}.
		\end{cases}
	\end{align}
	To show the lemma, we need an expansion for $\p_n$. We have, by \eqref{paper:eq:saddle-point} and again by \Cref{paper:C:Phi-Asymp},
	\begin{align*}
		-\Phi'_f(\varrho_n) = \frac{\omega_{\a} \Gamma(\a+1) \zeta(\a+1)}{\varrho_n^{\a+1}} + \frac{L_f(0)}{\varrho_n} + O\left(\varrho_n^{{R_0}-1}\right).
	\end{align*}
	By \Cref{paper:cor:rhosaddleexp1}, we have an expansion for $\p_n$ with an error $o(1)$. We iteratively find the first terms. By \Cref{paper:cor:rhosaddleexp1} we have $\p_n\sim a_{-\Phi_f',1}^\frac{1}{\a+1}n^{-\frac{1}{\a+1}}$, as $n\to\infty$. We next determine the second order term in $\p_n=\frac{a_{-\Phi_f',1}^\frac{1}{\a+1}}{n^\frac{1}{\a+1}}+\frac{K_2}{n^{\k_2}}+o(n^{-\k_2})$ for some $\k_2<\frac{1}{\a+1}$ and $K_2\in\C$. We choose $\k$ in
	\begin{align*}
		n \left( 1 + \frac{K_2}{a_{-\Phi_f',1}^{\frac{1}{\a+1}}n^{\kappa_2 - \frac{1}{\a+1}}}\right)^{-\a-1} + \frac{L_f(0)}{a_{-\Phi_f',1}^{\frac{1}{\a+1}}} n^{\frac{1}{\a+1}} \left( 1 + \frac{K_2}{a_{-\Phi_f',1}^{\frac{1}{\a+1}} n^{\kappa_2 - \frac{1}{\a+1}}}\right)^{-1} = n + O(n^\kappa)
	\end{align*}
	as small as possible. One finds that
	\begin{align*}
		\frac{(\a+1)K_2}{a_{-\Phi_f',1}^{\frac{1}{\a+1}}} n^{1-\kappa_2+\frac{1}{\a+1}} = \frac{L_f(0)}{a_{-\Phi_f',1}^{\frac{1}{\a+1}}}n^{\frac{1}{\a+1}},
	\end{align*}
	and hence
	\begin{align} \label{paper:eq:Varrho-Approx}
		\varrho_n = \frac{a_{-\Phi_f',1}^{\frac{1}{\a+1}}}{n^{\frac{1}{\a+1}}} + \frac{L_f(0)}{(\a+1)n} + o\left( \frac{1}{n}\right).
	\end{align}
	Plugging \eqref{paper:eq:Varrho-Approx} into $\Phi_f$ leads, by \eqref{paper:eq:Phi-simple}, to 
	\[
		\Phi_f\left(\frac{a_{-\Phi_f',1}^\frac{1}{\a+1}}{n^\frac{1}{\a+1}}+\frac{L_f(0)}{(\a+1)n}+o\left(\frac1n\right)\right) = \frac{a_{-\Phi_f',1}^\frac{1}{\a+1}}{\a}n^\frac{\a}{\a+1} - \frac{L_f(0)}{\a+1} - L_f(0)\log(\p_n) + L_f'(0) + o(1).
	\]
	As a result, using \eqref{paper:eq:Varrho-Approx}, we conclude the claim.
\end{proof}

We are now ready to prove \Cref{paper:T:main2}.

\begin{proof}[Proof of Theorem \ref{paper:T:main2}]
	\Cref{paper:T:main2} follows from Lemmas \ref{paper:Prop:Algebras} \ref{paper:Prop:Algebras:2}, \ref{paper:L:generalminorarcsbound}, \ref{paper:L:InCalc2}, \ref{paper:L:Gfasy}, \ref{paper:L:mainasy}, \ref{paper:lem:One-Pole-Saddle} and Corollaries \ref{paper:cor:rhosaddleexp1} and \ref{paper:C:expTimesGf}.
\end{proof}

\subsection{The case of two positive poles of $L_f$}\label{paper:subsec:TwoPoles}

If $\a>0$ is the only positive pole of $L_f$, then we can calculate the single term in the exponential in the asymptotic of $p_f(n)$ explicitly, by \Cref{paper:T:main2}. In this subsection we assume that $L_f$ has exactly two positive simple poles, $\a$ and $\b$. In this case, \Cref{paper:C:Phi-Asymp} with $k=1$ gives
\begin{equation*}%\label{paper:eq:phidouble}
	-\Phi'_f(z) = \frac{c_1}{z^{\a+1}} + \frac{c_2}{z^{\b+1}} + \frac{c_3}{z} + O_R\left(|z|^{R_0-1}\right)
\end{equation*}
with $R_0$ from \eqref{paper:eq:R0defi}.
Above we set $c_j:=a_{-\Phi_f',j}$ for $1\le j\le3$, i.e., by Lemma \ref{paper:C:Phi-Asymp}
\begin{align} \label{paper:eq:cDef}
	c_1 = \omega_{\a} \Gamma(\a+1)\z(\a+1), \quad c_2 = \omega_{\b} \Gamma(\b+1) \z(\b+1), \quad c_3 = L_f(0).
\end{align}
In the next lemma, we approximate the saddle point in this special situation. 

\begin{lem}\label{paper:L:varrho}
	Let $f$ satisfy the conditions of \Cref{paper:T:main2}. Additionally assume that $L_f$ has exactly two positive poles $\a$ and $\b$ that satisfy $\frac{\ell+1}{\ell}\b<\a\le\frac{\ell}{\ell-1}\b$ for some $\ell\in\N$, where we treat the case $\ell=1$ simply as $2\b<\a$. Then there exists $0<r\le\frac{R}{\a+1}$ such that
	\begin{equation}\label{paper:eq:rhoexplicit}
		\p_n = \sum_{j=1}^{\ell+1} \frac{K_j}{n^{(j-1)\left(1-\frac{\b+1}{\a+1}\right)+\frac{1}{\a+1}}} + \frac{c_3}{(\a+1)n} + O_R\left(n^{-r-1}\right)
	\end{equation}
	for some constants $K_j$ independent of $n$ and $c_3$ as in \eqref{paper:eq:cDef}. In particular, we have
	\begin{align*}
		K_1 &= c_1^\frac{1}{\a+1},
		\hspace{3mm} K_2 = \frac{c_2}{(\a+1)c_1^\frac{\b}{\a+1}},
		\hspace{3mm} K_3 = \frac{c_2^2(\a-2\b)}{2(\a+1)^2c_1^\frac{2\b+1}{\a+1}},
		\hspace{3mm}
		K_4 = \frac{c_2^3\left(2\a^2-9\a\b-2\a+9\b^2+3\b\right)}{6(\a+1)^3c_1^\frac{3\b+2}{\a+1}}, \\
		K_5 &= \frac{c_2^4(6 \a^3-44 \a^2 \b-15 \a^2+96 \a \b^2+56 \a \b+6 \a-64 \b^3-48 \b^2-8 \b)}{24(\a+1)^4 c_1^{\frac{4\b + 3}{\a + 1}}}.
	\end{align*}
\end{lem}

\begin{proof}
	By \Cref{paper:cor:rhosaddleexp1}, the exponents of $\p_n$ that are at most $1$ are given by combinations
	\[
		\frac{1}{\a+1} + (j-1)\left(1-\frac{\b+1}{\a+1}\right) + m\left(1-\frac{1}{\a+1}\right) \le 1,
	\]
	with $j\in\N$ and $m\in\N_0$. A straightforward calculation shows that $\frac{\ell+1}{\ell}\b<\a\le\frac{\ell}{\ell-1}\b$ if and only if
	\[
		0 < \frac{1}{\a+1} + (j-1)\left(1-\frac{\b+1}{\a+1}\right) \le 1
	\]
	for all $1\le j\le\ell+1$ but not for $j>\ell+1$. Together with the error term induced by \Cref{paper:cor:rhosaddleexp1}, \eqref{paper:eq:rhoexplicit} follows. Assuming $\ell\ge5$, $K_1$ to $K_5$ and the term $\frac{c_3}{(\a+1)n}$ can be determined iteratively.
\end{proof}

We are now ready to prove asymptotic formulas if $L_f$ has exactly two positive poles.

\begin{thm}\label{paper:T:TwoPoleAsymptotics} 
	Assume that $f:\N\to\N_0$ satisfies the conditions of Theorem \ref{paper:T:main2} and that $L_f$ has exactly two positive poles $\a>\b$, such that $\frac{\ell+1}{\ell}\b<\a\le\frac{\ell}{\ell-1}\b$ for some $\ell\in\N$. Then we have
	\begin{multline*}
		p_f(n) = \frac{C}{n^b}\exp\left(A_1n^\frac{\a}{\a+1}+A_2n^\frac{\b}{\a+1}+\sum_{k=3}^{\ell+1} A_kn^{\frac{(k-1)\b}{\a+1}+\frac{k-2}{\a+1}+2-k}\right)\\
		\times \left(1+\sum_{j=2}^{M_1} \frac{B_j}{n^{\nu_j}} + O_{L,R}\left(n^{-\min\left\{\frac{2L-\a}{2(\a+1)},\frac{R}{\a+1}\right\}}\right)\right),\qquad (n\to\infty),
	\end{multline*}
	with
	\begin{equation}\label{paper:eq:A1}
		A_1 := (\w_{\a}\GG(\a+1)\z(\a+1))^\frac{1}{\a+1}\left(1+\frac1\a\right),\qquad A_2 := \frac{\w_{\b}\GG(\b)\z(\b+1)}{(\w_\a\GG(\a+1)\z(\a+1))^\frac{\b}{\a+1}},
	\end{equation}
	and for all $k \geq 3$
	\begin{multline*}%\label{paper:eq:Ak}%\label{paper:eq:C}\label{paper:eq:b}
		A_k := K_k + \frac{c_1^\frac{1}{\a+1}}{\a}\sum_{m=1}^\ell \binom{-\a}{m}\sum_{\substack{0\le j_1,\dots,j_\ell\le m\\j_1+\ldots+j_\ell=m\\j_1+2j_2+\ldots+\ell j_\ell=k-1}} \binom{m}{j_1,j_2,\dots,j_\ell}\frac{K_2^{j_1}\cdots K_{\ell+1}^{j_\ell}}{c_1^\frac{m}{a+1}}\\
		+ \frac{c_2}{\b c_1^\frac{\b}{a+1}}\sum_{m=1}^\ell \binom{-\b}{m}\sum_{\substack{0\le j_1,\dots,j_\ell\le m\\j_1+\ldots+j_\ell=m\\j_1+2j_2+\ldots+\ell j_\ell=k-2}} \binom{m}{j_1,j_2,\dots,j_\ell}\frac{K_2^{j_1}\cdots K_{\ell+1}^{j_\ell}}{c_1^\frac{m}{a+1}}.
	\end{multline*}
	Here, $C$ and $b$ are defined in \eqref{paper:eq:mainconstants2} and \eqref{paper:eq:b}, the $\nu_j$ run through $\CM+\CN$, the $K_j$ are given in \Cref{paper:L:varrho}, and $c_1$, $c_2$, and $c_3$ run through \eqref{paper:eq:cDef}.
\end{thm}

\begin{proof}
	Assume that $g:\N\to\C$ has an asymptotic expansion as $n\to\infty$ and denote by $[g(n)]_*$ the part with nonnegative exponents. With Lemmas \ref{paper:C:Phi-Asymp} and \ref{paper:L:mainasy} we obtain, using that $L_f$ has exactly two positive poles in $\a$ and $\b$,
	\begin{align*}
		p_f(n) & = \frac{C}{n^b} \exp\left( \left[n\varrho_n + \frac{c_1}{\a \varrho_n^\a} + \frac{c_2}{\b \varrho_n^\b}\right]_* %+ O\left(\p_n^{R_0}\right)
		\right) \left(1+\sum\limits_{j=2}^{M_1} \frac{a_j }{n^{\d_j}} + O_{L,R}\left(n^{-\min\left\{\frac{2L-\a}{2(\a+1)}, \frac{R}{\a+1}\right\}}\right)\right)
	\end{align*}
	with the $\d_j$ running through $\CM$. With the Binomial Theorem and \Cref{paper:L:varrho}, we find
	\begin{equation}\label{paper:eq:series}
		\hspace{-.1cm}\frac{c_1}{\a\p_n^\a} = \frac{c_1^\frac{1}{\a+1}}{\a}n^\frac{\a}{\a+1}\left(1 + \hspace{-.1cm}\sum_{m\ge1}\hspace{-.1cm} \binom{-\a}{m} \left(\sum_{j=2}^{\ell+1} \frac{K_jc_1^{-\frac{1}{\a+1}}}{n^{(j-1)\left(1-\frac{\b+1}{\a+1}\right)}} + \frac{c_3c_1^{-\frac{1}{\a+1}}}{(\a+1)n^\frac{\a}{\a+1}} + o\left(n^{-\frac{\a}{\a+1}}\right)\right)^m\hspace{.1cm}\right).
	\end{equation}
	By definition, $[\frac{c_1}{\a\p_n^\a}]_*$ is the part of the expansion of $\frac{c_1}{\a\p_n^\a}$ involving nonnegative powers of $n$, %as the remaining part tends to zero for $n \to \infty$. Therefore
	i.e., for $m\ge2$ in the sum on the right of \eqref{paper:eq:series} we can ignore the term
	\[
		\frac{c_3}{(\a+1)c_1^\frac{1}{\a+1}n^\frac{\a}{\a+1}} + o\left(n^{-\frac{\a}{\a+1}}\right).
	\]
	Applying the Multinomial Theorem to \eqref{paper:eq:series} gives
	\begin{multline}\label{eq:mult1}
		\frac{c_1}{\a\p_n^\a} = \frac{c_1^\frac{1}{\a+1}}{\a}n^\frac{\a}{\a+1} - \frac{c_3}{\a+1} + \frac{c_1^\frac{1}{\a+1}}{\a}\sum_{m=1}^\ell \binom{-\a}{m}\sum_{\substack{0\le j_1,j_2,\dots,j_\ell\le m\\j_1+\dots+j_\ell=m}} \binom{m}{j_1,j_2,\dots,j_\ell}\frac{K_2^{j_1}\cdots K_{\ell+1}^{j_\ell}}{c_1^\frac{m}{a+1}}\\
		\times n^{\frac{(j_1+2j_2+\dots+\ell j_\ell)\b}{\a+1}+\frac{j_1+2j_2+\dots+\ell j_\ell-1}{\a+1}-(j_1+2j_2+\dots+\ell j_\ell-1)} + o(1).
	\end{multline}
	Similarly, we have
	\begin{multline}\label{eq:mult2}
		\frac{c_2}{\b\p_n^\b} = \frac{c_2}{\b c_1^\frac{\b}{a+1}}n^\frac{\b}{\a+1} + \frac{c_2}{\b c_1^\frac{\b}{a+1}}\sum_{m=1}^\ell \binom{-\b}{m}\sum_{\substack{0\le j_1,j_2,\dots,j_\ell\le m\\j_1+\dots+j_\ell=m}} \binom{m}{j_1,j_2,\dots,j_\ell}\frac{K_2^{j_1}\cdots K_{\ell+1}^{j_\ell}}{c_1^\frac{m}{a+1}}\\
		\times n^{\frac{(j_1+2j_2+\dots+\ell j_\ell+1)\b}{\a+1}+\frac{j_1+2j_2+\dots+\ell j_\ell}{\a+1}-(j_1+2j_2+\dots+\ell j_\ell)} + o(1).
	\end{multline}
	Finally, we obtain, with Lemma \ref{paper:L:varrho},
	\begin{equation}\label{eq:mult3}
		[n\p_n]_*= K_1n^\frac{\a}{\a+1} + \sum_{m=1}^\ell K_{m+1}n^{\frac{m\b}{\a+1}+\frac{m-1}{\a+1}-(m-1)} +\frac{c_3}{\a+1}.
	\end{equation}
	Combining \eqref{eq:mult1}, \eqref{eq:mult2}, and \eqref{eq:mult3}, we find that
	\[
		\left[n\p_n+\frac{c_1}{\a\p_n^\a}+\frac{c_2}{\b\p_n^\b}\right]_* = \left(1+\frac1\a\right)c_1^\frac{1}{\a+1}n^\frac{\a}{\a+1} + \frac{c_2}{\b c_1^\frac{\b}{\a+1}}n^\frac{\b}{\a+1} + \sum_{k=2}^\ell A_{k+1}n^{\frac{k\b}{\a+1}+\frac{k-1}{\a+1}-(k-1)},
	\]
	where
	\begin{multline*}
		A_k = K_k + \frac{c_1^\frac{1}{\a+1}}{\a}\sum_{m=1}^\ell \binom{-\a}{m}\sum_{\substack{0\le j_1,j_2,\dots,j_\ell\le m\\j_1+\dots+j_\ell=m\\j_1+2j_2+\dots+\ell j_\ell=k-1}} \binom{m}{j_1,j_2,\dots,j_\ell}\frac{K_2^{j_1}\cdots K_{\ell+1}^{j_\ell}}{c_1^\frac{m}{a+1}}\\
		+ \frac{c_2}{\b c_1^\frac{\b}{a+1}}\sum_{m=1}^\ell \binom{-\b}{m}\sum_{\substack{0\le j_1,j_2,\dots,j_\ell\le m\\j_1+\dots+j_\ell=m\\j_1+2j_2+\dots+\ell j_\ell=k-2}} \binom{m}{j_1,j_2,\dots,j_\ell}\frac{K_2^{j_1}\cdots K_{\ell+1}^{j_\ell}}{c_1^\frac{m}{a+1}}.
	\end{multline*}
	Note that we have by definition of $c_1$, $c_2$ (see \eqref{paper:eq:cDef}), $K_1$, and $K_2$ (see \Cref{paper:L:varrho}),
	\begin{align*}
		A_1 &= \left(1+\frac1\a\right)c_1^\frac{1}{\a+1} = \left(1+\frac1\a\right)(\w_\a\GG(\a+1)\z(\a+1))^\frac{1}{\a+1},\\
		A_2 &= \frac{c_2}{\b c_1^\frac{\b}{a+1}} = \frac{\w_\b\GG(\b)\z(\b+1)}{(\w_\a\GG(\a+1)\z(\a+1))^\frac{\b}{\a+1}},
	\end{align*}
	which gives \eqref{paper:eq:A1}. Hence we indeed obtain, as $n\to\infty$, for suitable $M_1\in\N$
	\begin{align*}
		p_f(n) & = \frac{C}{n^b}\exp\left( A_1 n^{\frac{\a}{\a+1}} + A_2 n^{\frac{\b}{\a+1}} + \sum_{k=3}^{\ell+1} A_k n^{\frac{(k-1)\b}{\a+1}+\frac{k-2}{\a+1}-(k-2)} \right) \\
		& \hspace{4cm} \times \left(1+\sum\limits_{j=2}^{M_1} \frac{B_j}{n^{\nu_j}} + O_{L,R}\left(n^{-\min\left\{\frac{2L-\a}{2(\a+1)}, \frac{R}{\a+1}\right\}}\right)\right), 
	\end{align*}
	where the $\nu_j$ run, as in Theorem \ref{paper:T:main2}, through $\CM+\CN$. This proves the theorem. 
\end{proof}

\section{Proofs of Theorems \ref{paper:T:n-gonalpartitions}, \ref{paper:T:shortermainSO5}, and \ref{paper:T:rsu5asy}}\label{sec:proofs1.1-1.3}

We require the zeta function associated to a polynomial $P$,
\begin{align*}
 Z_P(s) := \sum_{n\geq1} \frac{1}{P(n)^s}
\end{align*}
with $P(n)>0$ for $n\in\N$. In particular, we consider $P=P_k$, where 
\begin{align*}
 P_k(w) := \frac{(k-2)w^2 - (k-4)w}{2}.
\end{align*}
The following lemma ensures that all the $P_k$ satisfy \ref{paper:main:1} with $L$ arbitrary large.

\begin{lem}\label{paper:lem:Lambdak}
	Let $k \geq 3$ be an integer and let
	\begin{align*}
		\Lambda^{[k]} := \left\{ P_k(n) : n \in \N\right\}.
	\end{align*}
	For every prime $p$, we have $|\LL^{[k]}\sm(\LL^{[k]}\cap p\N)|=\infty$.
\end{lem}

We next show that \ref{paper:main:3} and \ref{paper:main:4} hold.

\begin{prop}\label{paper:lem:zetaPk}
	Let $k\in\N$ with $k\ge3$.
	\begin{enumerate}[leftmargin=*,label=\rm{(\arabic*)}]
		\item\label{paper:lem:zetaPk:1} The function $Z_{P_k}$ has a meromorphic continuation to $\C$ with at most simple poles in $\frac12-\N_0$. The positive pole lies in $s=\frac12$.
		
		\item\label{paper:lem:zetaPk:3} We have $Z_{P_k}(s) \ll Q_k(|\im (s)|)$ as $|\im (s)|\to\infty$ for some polynomial $Q_k$.
	\end{enumerate}
\end{prop}

\begin{proof}
	(1) The meromorphic continuation of $Z_{P_k}$ to $\C$ follows by \cite[Theorem B]{MatWen}. By \cite[Theorem A (ii)]{MatWen} the only possible poles (of order at most one) are located at $\frac12-\frac12\N_0$. Holomorphicity in $-\N_0$ is a direct consequence of \cite[Theorem C]{MatWen}. Finally, note that $P_k(n) \ll_k n^2$. Thus, as $x\to\infty$,
	\begin{align*}
		\sum_{1\leq n \leq x} \frac{1}{P_k(n)^{\frac{1}{2}}} \gg_k \sum_{1\leq n \leq x} \frac{1}{n}.
	\end{align*}
	This proves the existence of a pole in $s = \frac12$, completing the proof.\\
	(2) This result follows directly by \cite[Proposition 1 (iii)]{MatWen}.
\end{proof}

To apply Theorem \ref{paper:T:main2}, it remains to compute $Z_{P_k}(0)$ and $Z_{P_k}'(0)$, as well as $\Res_{s=\frac12} Z_{P_k}(s)$.

\begin{prop}\label{paper:lem:zetapkRes}
	Let $k\in\mathbb{N}$ with $k\geq3$.
	\begin{enumerate}[leftmargin=*,label=\rm{(\arabic*)}]
		\item\label{paper:lem:zetapkRes:1} We have $Z_{P_k}(0)= \frac{1}{2-k}$ and
		\[
			Z_{P_k}'(0) = \frac{\log\left(\frac{k-2}{2}\right)}{k-2} + \log\left(\GG\left(\frac{2}{k-2}\right)\right) - \log(2\pi).
		\]
		
		\item\label{paper:lem:zetapkRes:2} We have $\Res_{s=\frac12}Z_{P_k}(s)=\sqrt{\frac{1}{2(k-2)}}$. 
 \end{enumerate}
\end{prop}

\begin{proof}
	\ref{paper:lem:zetapkRes:1} Since the roots of $P_k$ are not in $\R_{\ge1}$, we may use \cite[Theorem D]{MatWen} to obtain that $Z_{P_k}(0)=\frac{1}{2-k}$. For the derivative, one applies \cite[Theorem E]{MatWen} yielding
	\[
		Z_{P_k}'(0) = \frac{\log\left(\frac{k-2}{2}\right)}{k-2} + \log\left(\GG\left(\frac{2}{k-2}\right)\right) - \log(2\pi).
	\]
	%with \textit{Hurwitz' multiple zeta function}
	%\begin{align*}
		% \zeta(s_1,\dots,s_d;\theta_1,\dots,\theta_d) := \sum\limits_{m_1>\dots>m_d>0} (m_1+\theta_1)^{-s_1}\cdots(m_d+\theta_d)^{-s_d}.
	%\end{align*}
	
	%Since $\z(s)$ and $\z(s;\th)$ for all $\th\in\C\setminus(-\N)$ both are holomorphic in $s=\frac12$, also their product is. %and hence
	%\begin{align*}
		% \Res_{s=\frac{1}{2}} Z_{P_k}(s) = -\left(\frac{2}{k-2}\right)^\frac{1}{2}\Res_{s=\frac{1}{2}} \left(\zeta\left(s,s;0,-\frac{k-4}{k-2}\right) + \zeta\left(s,s;-\frac{k-4}{k-2},0\right)\right).
	%\end{align*}
	%Using \cite[Theorem 1.4]{KelMas}, yields
	%\begin{align*}
		% \Res_{s=\frac{1}{2}} \zeta(s,s;\theta_1,\theta_2) = \theta_1-\theta_2 -\frac{1}{2}
	%\end{align*}
	%and hence we have
	\noindent\ref{paper:lem:zetapkRes:2} Since $Z_{P_k}(s)=(\frac{2}{k-2})^s\sum\limits_{n\ge1}(n-\frac{k-4}{k-2})^{-s}n^{-s}$, the result follows as the sum has residue $\frac12$ at $s=\frac12$ by equation (16) of \cite{MatWen}.
\end{proof}

The previous three lemmas are used to prove Theorem \ref{paper:T:n-gonalpartitions}.

%\begin{thm} \label{paper:T:asymptrisqupen}
%	We have for every $N\in\N$
%	\begin{multline*}
%		p_k(n) = \frac{\left(\frac{k-2}{2}\right)^\frac{1}{k-2}\GG\left(\frac{2}{k-2}\right) \left(\sqrt{\frac{1}{2(k-2)}\pi}\z\left(\frac32\right)\right)^\frac{k}{3k-6}} {2\sqrt{3}\pi^\frac32n^\frac{5k-9}{6k-12}}\\
%		\times \exp\left(3\left(\sqrt{\tfrac{1}{2(k-2)}\pi}\z\left(\tfrac32\right)\right)^\frac23 n^\frac13\right)\left(1+\sum_{j=1}^N \frac{B_j}{n^{\frac{j}{3}}}+O_N\left(n^{-\frac{N+1}{3}}\right)\right),
%	\end{multline*}
%	where the coefficients $B_j$ can be calculated explicitly. In particular, one obtains as $n\to\infty$
%	\begin{align*}
%		p_{k}(n) \sim \frac{\left(\frac{k-2}{2}\right)^\frac{1}{k-2}\GG\left(\frac{2}{k-2}\right)\left(\sqrt{\frac{1}{2(k-2)}\pi}\zeta\left(\frac{3}{2}\right)\right)^\frac{k}{3k-6}}{2\sqrt{3}\pi^\frac{3}{2}n^{\frac{5k-9}{6k-12}}} \exp\left(3\left(\sqrt{\frac{1}{2(k-2)}\pi}\zeta\left(\tfrac{3}{2}\right)\right)^\frac{2}{3} n^\frac{1}{3}\right).
%	\end{align*}
%	and 
%\end{thm}

\begin{proof}[Proof of Theorem \ref{paper:T:n-gonalpartitions}]
	We may apply \Cref{paper:T:main2} as \Cref{paper:lem:Lambdak} and \Cref{paper:lem:zetaPk} ensure that conditions \ref{paper:main:1}--\ref{paper:main:4} are satisfied. Hence, one obtains an asymptotic formula for $p_k(n)$. The constants occurring in Theorem \ref{paper:T:main2} are computed using \eqref{paper:eq:mainconstants2}, \eqref{paper:eq:b}, and \Cref{paper:lem:zetapkRes}. That the exponential consists only of the term $A_1 n^\frac{1}{3}$ follows by \Cref{paper:T:main2}, since $Z_{P_k}(s)$ has exactly one positive pole, lying in $s=\frac12$. Note that we are allowed to choose $L$ and $R$ arbitrarily large due to \Cref{paper:lem:Lambdak} and \Cref{paper:lem:zetaPk} \ref{paper:lem:zetaPk:1}. 
\end{proof}

We consider some special cases of \Cref{paper:T:n-gonalpartitions}. 

\begin{cor}
\label{paper:cor:n-gonalpartitions}
 For triangular numbers, squares, and pentagonal numbers, respectively, we have
	\begin{align*}
		p_3(n) &\sim \frac{\z\left(\frac32\right)}{2^\frac72\sqrt3\pi n^\frac32}\exp\left(\frac32\pi^\frac13\z\left(\frac32\right)^\frac23n^\frac13\right),\qquad p_4(n) \sim \frac{\z\left(\frac32\right)^\frac23}{2^\frac73\sqrt3\pi^\frac76n^\frac76} \exp\left(\frac{3}{2^\frac43}\pi^\frac13\z\left(\frac32\right)^\frac23n^\frac13\right),\\
		p_5(n) &\sim \frac{\GG\left(\frac23\right)\z\left(\frac32\right)^\frac59} {2^\frac{13}{6}3^\frac49\pi^\frac{11}{9}n^\frac{19}{18}} \exp\left(\frac{3^\frac23}{2}\pi^\frac13\z\left(\frac32\right)^\frac23n^\frac13\right).
	\end{align*}
\end{cor}

The next lemma shows that $\prod_{j,k\ge1}({1-q^\frac{jk(j+k)(j+2k)}{6}})^{-1}$ satisfies \ref{paper:main:1} for $L$ arbitrarily large.% of Theorem \ref{paper:T:main2}.

\begin{lem}\label{paper:lem:SO5conditions}
	Let $f \colon \N \to \N_0$ be defined by
	\[
		f(n) := \left|\left\{(j,k)\in\N^2 : \frac{jk(j+k)(j+2k)}{6}=n\right\}\right|.
	\]
	Then, for all primes $p$, we have $|\Lambda \setminus (\Lambda\cap p\N )| = \infty$.
\end{lem}

For investigating the function $\z_{\SO (5)}$, we need the \textit{Mordell--Tornheim zeta function}, defined by
\begin{align*}
 \zeta_{\MT,2}(s_1,s_2,s_3) := \sum_{m,n\geq1} m^{-s_1} n^{-s_2} (m+n)^{-s_3}.
\end{align*}
By \cite{Mat}, for $\re(s)>1$ and some $-\re(s)<c<0$ we get a relation between $\z_{\MT,2}$ and $\z_{\SO(5)}$ via
\begin{align} \label{paper:eq:zetaso5}
 \zeta_{\SO(5)}(s) = \frac{6^s}{2\pi i \Gamma(s)} \int_{c - i\infty}^{c + i\infty} \Gamma(s+z)\Gamma(-z) \zeta_{\MT,2}(s,s-z,2s+z)dz.
\end{align}
We have the following theorem.

\begin{thm}[\cite{Mat2}, Theorem 1] \label{paper:T:MT2} The function $\zeta_{\MT, 2}$ has a meromorphic continuation to $\C^3$ and its singularities satisfy
$
 s_1 + s_3 = 1-\ell, 
 s_2 + s_3 = 1-\ell, 
 s_1 + s_2 + s_3 = 2,
$
with $\ell \in \N_0$.
\end{thm}

Fix $M\in\N_0$ and $0<\e<1$. Let $\re(s_1),\re(s_3)>1$, $\re(s_2)>0$, and $s_2\notin\N$. Then, for $\re(s_2)<M+1-\e$, we have (see equation (5.3) in \cite{Mat2})
\begin{align}\nonumber
	\zeta_{\MT,2}(s_1, s_2, s_3) & = \frac{\Gamma(s_2+s_3-1)\Gamma(1-s_2)}{\Gamma(s_3)} \zeta(s_1 + s_2 + s_3 - 1) \\
	\nonumber
	& \hspace{0.35cm} + \sum_{m=0}^{M-1} {-s_3 \choose m} \zeta(s_1+s_3+m)\zeta(s_2 - m) \\
	\label{paper:eq:MTintegral}
	&\quad + \frac{1}{2\pi i} \int_{M-\varepsilon - i\infty}^{M-\varepsilon + i\infty} \frac{\Gamma(s_3 + w)\Gamma(-w)}{\Gamma(s_3)} \zeta(s_1 + s_3 + w)\zeta(s_2 - w)dw.
\end{align}
The first two summands on the right-hand side of \eqref{paper:eq:MTintegral} extend meromorphically to $\C^3$, so to show that \eqref{paper:eq:zetaso5} extends meromorphically, we consider \eqref{paper:eq:MTintegral}. Note that $\mathrm{Re}(w)=M-\varepsilon$. To avoid poles on the line of integration, we assume that 
\begin{align}\label{paper:eq:Conditions1}
	\mathrm{Re}(s_3 + w) > 0 & \Leftrightarrow \mathrm{Re}(s_3) > \varepsilon-M, \\
	\label{paper:eq:Conditions2}
	\mathrm{Re}(s_1+s_3+w) > 1 & \Leftrightarrow \mathrm{Re}(s_1) + \mathrm{Re}(s_3) > 1 - M + \varepsilon, \\
	\label{paper:eq:Conditions3}
	\mathrm{Re}(s_2 - w) < 1 & \Leftrightarrow \mathrm{Re}(s_2) < 1 + M - \varepsilon.
\end{align}
Note that the final condition is already assumed above.

By Propostition 2.6 \ref{paper:GC:4}, the integral converges compactly and the integrands are locally holomorphic. Thus, the integral is a holomorphic function in the region defined by \eqref{paper:eq:Conditions1}, \eqref{paper:eq:Conditions2}, and \eqref{paper:eq:Conditions3}. Recalling \eqref{paper:eq:zetaso5}, we are interested in $\z_{\MT,2}(s,s-z,2s+z)$. By \Cref{paper:T:MT2}, this function is meromorphic in $\C^2$ and holomorphic outside the hyperplanes defined by $3s+z=1-\ell$, $3s=1-\ell$, and $4s=2$, where $\ell\in\N_0$. With \eqref{paper:eq:MTintegral}, we obtain
\begin{multline}\label{paper:eq:MTZeta-Formula}
	\z_{\MT,2}(s,s-z,2s+z) = \frac{\GG(3s-1)\GG(z+1-s)}{\GG(2s+z)}\z(4s-1)\\
	+ \sum_{m=0}^{M-1} \binom{-2s-z}{m}\z(3s+z+m)\z(s-z-m) + I_M(s;z),
\end{multline}
where $s\in\C\setminus\{\frac12,\frac{1-\ell}{3}\}$, and
\begin{equation*}%\label{eq:IMDef}
	I_M(s; z) := \frac{1}{2\pi i} \int_{M-\varepsilon-i\infty}^{M-\varepsilon+i\infty} \frac{\Gamma(2s+z+w)\Gamma(-w)}{\Gamma(2s+z)} \zeta(3s+z+w)\zeta(s-z-w) dw.
\end{equation*}
The following lemma shows that $I_M(s;z)$ is holomorphic in $z$. To state it let
\[
	\mu = \mu_{M,\s} := \max\{-1+\s-M+\e,1-3\s-M+\e,-2\s-M+\e\}.
\]

\begin{lem}\label{paper:L:integralholo}
	Let $s=\s+it\in\C$, $M\in\N_0$, and $0<\e<1$. Then $z\mapsto I_M(s;z)$ is holomorphic in $S_{\mu,\infty}$.
\end{lem}

\begin{proof}
	If $z\in S_{\mu,\infty}$, then $\re(2s+z+w)>0$, $\re(3s+z+w)>1$, and $\re(s-z-w)<1$ for $w\in\C$ satisfying $\re(w)=M-\e$, so $\GG(2s+z+w)$, $\z(3s+z+w)$, and $\z(s-z-w)$ have no poles on the path of integration. As $0<\e<1$, we have $M-\e\notin\N_0$, so $w\mapsto\GG(-w)$ has no pole if $\re(w)=M-\e$. As a result, no pole is located on the path of integration, and by \Cref{paper:GammaCollect} \ref{paper:GC:4} and the uniform polynomial growth of the zeta function along vertical strips we find that the integral converges uniformly on compact subsets of $S_{\mu,\infty}$.
\end{proof}

The next lemma shows, that $I_M$ is bounded polynomially in certain vertical strips. A proof is obtained using Propositions \ref{paper:GammaCollect} \ref{paper:GC:4} and \ref{paper:ZetaFunc} \ref{paper:ZF:3}.

\begin{lem}\label{paper:L:Ibound}
	Let $\s_1 < \s_2$ and $\s_3 < \s_4$, such that $S_{\s_3, \s_4} \subset S_{\mu,\infty}$ for all $s \in S_{\s_1, \s_2}$ and fix $0 < \varepsilon < 1$ sufficiently small. In $S_{\sigma_1,\sigma_2} \times S_{\sigma_3,\sigma_4}$ the function $(s,z) \mapsto I_M(s;z)$ is holomorphic and satisfies $|I_M(s;z)|\le P_{\s_1,\s_2, \s_3,\s_4,M}(|\im(s)|,|\im(z)|)$ for some polynomial $P_{\s_1,\s_2, \s_3,\s_4,M}(X,Y)\in\R[X,Y]$.
\end{lem}

Next we investigate $\z_{\MT,2}(s,s-z,2s+z)$ for fixed $s$ more in detail.

\begin{lem}\label{paper:MTfixsvarz}
	Let $s\in\C\sm\{\frac12,\frac13-\frac13\N_0\}$. Then $z\mapsto\z_{\MT,2}(s,s-z,2s+z)$ is holomorphic in the entire complex plane except for possibly simple poles in $z=1-\ell-3s$ with $\ell\in\N_0$.
\end{lem}

\begin{proof}
	As holomorphicity is a local property, it suffices to consider arbitrary right half-planes. By Lemma \ref{paper:L:integralholo}, for $M$ sufficiently large, $I_M$ is holomorphic in an arbitrary right half-plane. By \eqref{paper:eq:MTintegral}, possible poles of $\z_{\text{MT},2}(s,s-z,2s+z)$ therefore lie in $z=s-\ell$ and in $z=-3s-m-\ell$, $\ell\in\N$. A direct calculation shows that the residue at $z=s-\ell$ vanishes if $\ell\le M-1$. Consequently, for a fixed pole $s-\ell$, we can choose $M$ sufficiently large such that we only have to consider the of \eqref{paper:eq:MTintegral}. This gives the claim.
\end{proof}

We are now ready to prove growth properties of $\z_{\MT,2}$. As we need to avoid critical singular points, we focus on incomplete half-planes of the type $S_{\s_1, \s_2, \d}$ (with $\d>0$ arbitrarily small).

\begin{lem}\label{paper:L:zetaMTvertical}
	Let $\s_1<\s_2$, $\s_3<\s_4$ with $1-3\s_1<\s_3$ and $\d > 0$ arbitrarily small. For $(s,z)\in S_{\s_1,\s_2,\d}\times S_{\s_3,\s_4}$, we have, for some polynomial $P_{\s_1,\s_2,\s_3,\s_4,\d}$ only depending on $S_{\s_1,\s_2,\d}$ and $S_{\s_3,\s_4}$,
	\[
		|\z_{\MT,2}(s,s-z,2s+z)| \le P_{\s_1,\s_2,\s_3,\s_4,\d}(|\im(s)|,|\im(z)|).
	\]
	If $\s_1 < 0$, for all $s \in U$ with $U \subset S_{\s_1, \s_2}$, a sufficiently small neighborhood of $0$, we have
		\begin{align*}
	\left| \frac{\z_{\MT,2}(s,s-z,2s+z)}{\GG(s)}\right| \leq P_{\s_3, \s_4, U}(|\im(z)|),
	\end{align*}
where the polynomial $P_{\s_3, \s_4, U}$ only depends on $\s_3$, $\s_4$, and $U$.
\end{lem}

We need another lemma dealing with the poles of the Mordell--Tornheim zeta function.

\begin{lem}\label{paper:L:Holointegers}
	Let $k\in\N_0$. Then the meromorphic function $s\mapsto\z_{\MT, 2}(s,s-k,2s+k)$ is holomorphic for $s=-\ell$ with $\ell\in\N_{\ge\frac k2}$ and has possible simple poles at $s=\ell\in\N_0$ with $0\le\ell<\frac k2$. In particular, $s\mapsto{\GG(s+k)\z_{\MT,2}(s,s-k,2s+k)}{\GG(s)^{-1}}$ is holomorphic at $s=-\ell$ with $\ell\in\N_0$.
\end{lem}

\begin{proof}
	 Let $s$ lie in a bounded neighborhood of $-\ell$. We use \eqref{paper:eq:MTZeta-Formula} with $s=k$. Analogous to the proof of \Cref{paper:L:integralholo}, the function $s\mapsto I_M(s;k)$ is holomorphic in a neighborhood of $s=-\ell$. The analysis of the remaining terms is straightforward, and the lemma follows.
\end{proof}

The next lemma states where the integral of \eqref{paper:eq:zetaso5} defining $\z_{\SO(5)}$ is a meromorphic function.

\begin{lem}\label{paper:L:MTintholo}
	Let $\e>0$ be sufficiently small and let $K\in\N$. Then the function
	\begin{equation}\label{function}
		s \mapsto \frac{1}{2\pi i \Gamma(s)} \int_{K - \varepsilon - i\infty}^{K - \varepsilon + i\infty} \Gamma(s+z)\Gamma(-z) \zeta_{\MT,2}(s, s-z, 2s+z)dz
	\end{equation}
	is meromorphic on the half plane $\{s\in\C:\re(s)>\frac{1-K+\e}{3}\}$ with at most simple poles in $\{\frac12,\frac13-\frac13\N_0\}\setminus(-\mathbb{N}_0)$ (with $\re(s)>\frac{1-K+\e}{3}$) and grows polynomially on vertical strips with finite width.
\end{lem}
	
\begin{proof}
	We first show holomorphicity in $S_{\s_1,\s_2,\d}$ with $\frac{1-K+\e}{3}<\s_1<\s_2$ and $0<\d<1$. Since $\re(s)>\frac{1-K+\e}{3}>-K+\e$, there are no poles of $\GG(s+z)\GG(-z)$ on the path of integration $\re(z)=K-\e$. By \Cref{paper:MTfixsvarz}, $z\mapsto\z_{\MT,2}(s,s-z,2s+z)$ has no poles for $s\in S_{\s_1,\s_2,\d}$, as $\re(z+3s-1)=K-\e+3\re(s)-1>0$. By \Cref{paper:GammaCollect} \ref{paper:GC:4}, \Cref{paper:L:zetaMTvertical}, and \Cref{paper:L:IntegralPolyBound}, the integral is holomorphic away from singularities and grows polynomially on vertical strips of finite width.
	
	We are left to show that \eqref{function} has at most a simple pole at $s=s_0$, where $s_0\in\{\frac12,\frac13-\frac13\N_0\}\setminus(-\mathbb{N}_0)$ with $s_0\ge\frac{1-K+\e}{3}$. Recall the representation of $\z_{\MT,2}$ in \eqref{paper:eq:MTZeta-Formula}. By \Cref{paper:L:Ibound}
	\[
		\int_{K-\e-i\infty}^{K-\e+i\infty} \GG(s+z) \GG(-z) I_M(s;z) dz
	\]
	converges absolutely and uniformly on any sufficiently small compact subset $C$ containing $s_0$ for $M$ sufficiently large. Similarly, by Propositions \ref{paper:ZetaFunc} \ref{paper:ZF:3} and \ref{paper:GammaCollect} \ref{paper:GC:4},
	\[
		\int_{K-\e-i\infty}^{K-\e+i\infty} \GG(s+z)\GG(-z) \sum_{m=0}^{M-1} \binom{-2s-z}{m}\z(3s+z+m)\z(s-z-m) dz
	\]
	converges absolutely and uniformly in $C$. In particular, both integrals continue holomorphically to $s_0$. As $s\mapsto\frac{1}{\GG(s)}$ is entire, it is sufficient to study
	\[
		\frac{\GG(3s-1)\z(4s-1)}{\GG(s)} \int_{K-\e-i\infty}^{K-\e+i\infty} \frac{\GG(s+z)\GG(-z)\GG(1+z-s)}{\GG(2s+z)}dz
	\]
	around $s_0$. Again, by \Cref{paper:GammaCollect} \ref{paper:GC:4}, the integral converges absolutely and uniformly in $C$. As $\frac{\GG(3s-1)\z(4s-1)}{\GG(s)}$ has at most a simple pole in $s_0$ and a removable singularity if $s_0\in-\N_0$, the proof of the lemma is complete.
\end{proof}

The following lemma is a refinement of \Cref{paper:L:MTintholo} for the specific case that $z\in\Z$ and follows from \Cref{paper:L:Ibound}, by using Propositions \ref{paper:GammaCollect} and \ref{paper:ZetaFunc}.
 
\begin{lem}\label{paper:L:MTzetakpoly}
	Let $k\in\N_0$ with $0\le k\le K-1$. Then, for all $\s_1<\s_2$, there exists a polynomial $P_{K,\s_1,\s_2}$, such that, uniformly for all $\s_1\le\re(s)\le\s_2$ and $|\im(s)|\ge1$,
	\[
		|\z_{\MT,2}(s,s-k,2s+k)| \le P_{K,\s_1,\s_2}(|\im(s)|).
	\]
\end{lem}

The following theorem shows that the function $\z_{\SO(5)}$ satisfies the conditions of \Cref{paper:T:main2} and gives the more precise statement of \Cref{paper:T:shortermainSO5}.

\begin{thm} \label{paper:T:mainSO5} 
	The function $\z_{\SO(5)}$ extends to a meromorphic function in $\C$ and is holomorphic in $\N_0$. For $K\in\N$ and $0 < \varepsilon < 1$, we have, on $S_{\frac{1-K+\varepsilon}{3}, \infty}$,
	\begin{align}\label{paper:zeta0:1}
 	\zeta_{\SO(5)}(s) &= \frac{6^s}{\GG(s)} \sum_{k=0}^{K-1} \frac{(-1)^k \Gamma(s+k)}{k!}\zeta_{\MT,2}(s, s-k, 2s+k)\nonumber\\
 		& \hspace{1cm} + \frac{6^s}{2\pi i\GG(s)} \int_{K - \varepsilon - i\infty}^{K - \varepsilon + i\infty} \Gamma(s+z)\Gamma(-z) \zeta_{\MT,2}(s, s-z, 2s+z) dz.
	\end{align}
	All poles of $\z_{\SO(5)}$ are simple and contained in $\{\frac12,\frac13,-\frac13,-\frac23,\dots\}$. Furthermore, for all $\s_0\le\s\le\s_1$ as $|\im(s)|\to\infty$, for some polynomial depending only on $\s_0$ and $\s_1$,
	\begin{equation*}%\label{polbaud}
		|\z_{\SO(5)}(s)| \le P_{\s_0,\s_1}(|\im(s)|).
	\end{equation*}
\end{thm}

\begin{proof}
	Assume $\re(s) > 1$. By \Cref{paper:MTfixsvarz}, the only poles of the integrand in \eqref{paper:eq:zetaso5} in $S_{-\re(s),\infty}$ lie at $z\in\N_0$. By shifting the path to the right of $\re(z)=M-\e$, we find, with \Cref{paper:L:zetaMTvertical} and the Residue Theorem, that \eqref{paper:zeta0:1} holds on $S_{1,\infty}$. By \Cref{paper:L:MTintholo} the right-hand side is a meromorphic function on $S_{\frac{1-K+\e}{3},\infty}$. By \Cref{paper:T:MT2}, the functions $s\mapsto\z_{\MT,2}(s,s-k,2s+k)$ only have possible (simple) poles for $s_1+s_3=3s+k=1-\ell$,\ $s_2+s_3=3s =1-\ell$, $s_1+s_2+s_3=4s=2$, with $\ell\in\N_0$, i.e., for $s\in\{\frac12,\frac13,0,-\frac13,-\frac23,-1,\dots\}$. However, by \Cref{paper:L:Holointegers} the sum in \eqref{paper:zeta0:1} continues holomorphically to $-\N_0$, so the sum only contributes possible poles $s\in\CS:=\{\frac12,\frac13,-\frac13,-\frac23,-\frac43,\dots\}$. Note that this argument does not depend on the choice of $K$. On the other hand, if we choose $K$ sufficiently large, then the integral in \eqref{paper:zeta0:1} is a holomorphic function around $s=-m$ for fixed but arbitrary $m\in\N_0$, and it only contributes poles in $\CS$ in $S_{\frac{1-K+\e}{3},\infty}$ by \Cref{paper:L:MTintholo}, where $0<\e<1$. So the statement about the poles follows if $K\to\infty$.
	
	We are left to show the polynomial bound. With Lemma \ref{paper:L:MTzetakpoly} we obtain the bound for the finite sum, as we chose $K$ in terms of $\s_0$ and $\s_1$. Lemma \ref{paper:L:MTintholo} implies the polynomial bound for the integral.
\end{proof}

To apply \Cref{paper:T:main2} we require $\z_{\SO(5)}(0)$. %Note that for interchanging limit and integration, we use Lebesgue's Theorem of dominated convergence.

\begin{prop} \label{paper:P:zetaso5(0)}
We have $\zeta_{\SO(5)}(0) = \frac{3}{8}$.
\end{prop}

\begin{proof}
	Since $I_M(s; z)$ is holomorphic in $s$ for $z\in S_{\mu,\infty}$ by \Cref{paper:L:Ibound} and $\GG(s)$ has a pole in $s=0$,
	\begin{equation}\label{paper:eq:MTZeta-firstterm}
		\lim\limits_{s\rightarrow 0} \frac{I_M(s; z)}{\Gamma(s)} = 0.
	\end{equation}
	Let $K\in\N$. For $z\in\C$ with $\re(z)=K-\frac12$ and $m\in\N_0$, we have $\pm(z+m)\ne1$. Hence, $s\mapsto\binom{-2s-z}{m}\z(3s+z+m)\z(s-z-m)$ is holomorphic at $s=0$. This implies that for $z\in\C$ with $\re(z)=K-\frac12$, we have
	\begin{equation*}%\label{paper:eq:MTZeta-secondterm}
		\lim\limits_{s\rightarrow 0} \binom{-2s-z}{m}\frac{\zeta(3s+z+m)\zeta(s-z-m)}{\Gamma(s)} = 0.
	\end{equation*}
	Using this, \eqref{paper:zeta0:1} with $\e=\frac12$, \eqref{paper:eq:MTZeta-firstterm}, \Cref{paper:GammaCollect} \ref{paper:GC:5}, and Lebesgue's dominated convergence theorem, we obtain, for integers $K\ge3$,
	\[
		\lim_{s\to0} \frac{6^s}{2\pi i\GG(s)}\int_{K-\frac12-i\infty}^{K-\frac12+i\infty} \GG(s+z)\GG(-z)\z_{\MT,2}(s,s-z,2s+z) dz
		= \frac{i}{72} \int_{K-\frac12-i\infty}^{K-\frac12+i\infty} \frac{1}{\sin(\pi z)} dz.
	\]
	Since $\sin(\pi(z+1)) = -\sin(\pi z)$ and
	\begin{align*}
		\lim\limits_{L\rightarrow\infty}\int_{K-\frac{1}{2}-i L}^{K + \frac{1}{2}-i L} \frac{1}{\sin(\pi z)} dz = \lim\limits_{L\rightarrow\infty}\int_{K+\frac{1}{2}+i L}^{K - \frac{1}{2}+i L} \frac{1}{\sin(\pi z)} dz = 0,
	\end{align*}
	the Residue Theorem implies that
	\begin{equation}\label{paper:zeta0:2}
		\lim_{s\to0} \tfrac{6^s}{2\pi i\GG(s)}\int_{K-\frac12-i\infty}^{K-\frac12+i\infty} \GG(s+z)\GG(-z)\z_{\MT,2}(s,s-z,2s+z) dz = \tfrac{1}{72}\Res_{z=K} \tfrac{\pi}{\sin(\pi z)} = \tfrac{(-1)^K}{72}.
	\end{equation}
	In the following we use that $\zeta(s)$ does not have a pole in $s=\pm m$ for $m\in\N_{\geq 2}$, implying that $s\mapsto \binom{-2s-1}{m-1}\z(3s+m)\z(s-m)$ is holomorphic at $s=0$. Moreover $s\mapsto\GG(s+k)\binom{-2s-k}{m}\z(3s+k+m)\z(s-k-m)$ is holomorphic at $s=0$ for $(k,m)\in(\N\times\N_0)\backslash\{(1,0)\}$. Thus, using Propositions \ref{paper:GammaCollect} \ref{paper:GC:3} and \ref{paper:ZetaFunc} \ref{paper:ZF:4} and the fact that $\z(-1)=-\frac{1}{12}$ and $\z(0) = \frac12$, we obtain, with \eqref{paper:eq:MTZeta-Formula},
	\begin{align}\nonumber
		\lim\limits_{s\rightarrow 0} \frac{6^s}{\Gamma(s)} &\sum_{k=0}^{K-1} \frac{(-1)^k \Gamma(s+k)}{k!}\zeta_{\MT,2}(s, s-k, 2s+k)\\
		\label{paper:zeta0:3}
		=& \frac{3}{8}+\frac{(-1)^{K+1}}{72}+\lim\limits_{s\to 0} I_M(s;0) +\sum\limits_{k=1}^{K-1} \frac{(-1)^k}{k}\lim\limits_{s\to 0} \frac{I_M(s;k)}{\GG(s)}.
	\end{align}
	Since, by Lemma \ref{paper:L:Ibound}, $s\mapsto I_M(s;k)$ is holomorphic at $s=0$ for every $k\in\N_0$ and $\frac{1}{\GG(s)}$ vanishes in $s=0$, we have
	\begin{align*}
		\lim\limits_{s\rightarrow 0} \frac{I_M(s;k)}{\Gamma(s)} = 0.
	\end{align*}
	Applying the Lebesgue dominated convergence theorem gives $\lim\limits_{s\to0}I_M(s;0)=0$, yielding the claim with \eqref{paper:zeta0:1}, \eqref{paper:zeta0:2}, and \eqref{paper:zeta0:3}.
\end{proof}

Furthermore, we need certain residues of $\z_{\SO(5)}$.

\begin{prop}\label{paper:P:reszetaso(5)1/2}
	The poles of $\z_{\SO(5)}$ are precisely $\{\frac12\}\cup\{\frac d3\notin\Z:d\le 1\text{ odd}\}$. We have
	\[
		\Res_{s=\frac12} \z_{\SO(5)}(s) = \frac{\sqrt3\GG\left(\frac14\right)^2}{8\sqrt\pi}.
	\]
	Moreover for $d\in\Z_{\leq1}\setminus ( - 3\N_0)$, 
	\begin{equation}\label{paper:resu}
		\Res_{s=\frac d3} \z_{\SO(5)}(s) = \frac{3^{\frac d3-\frac32}\pi\GG\left(\frac d6\right)\z\left(\frac{4d}{3}-1\right)}{2^{\frac d3-1}(1-d)!\GG\left(\frac d3\right)^2\GG\left(\frac d2\right)}\left(\frac d3\right)\left(1+2^{\frac{2d}{3}-1}\right).
	\end{equation}
	In particular, we have
	\[
		\Res_{s=\frac13} \z_{\SO(5)}(s) = \frac{2^\frac13+1}{3^\frac23}\z\left(\frac13\right).
	\]
\end{prop}

\begin{proof}
	With Lemma \ref{paper:L:MTintholo}, near $s=\frac12$, we can choose $K=1$ in \eqref{paper:zeta0:1} and obtain 
	\begin{multline*}
		\Res_{s=\frac12}\z_{\SO(5)}(s)\\ = \lim_{s \to \frac12} \left( s - \frac12 \right) \left( 6^s \zeta_{\mathrm{MT},2}(s,s,2s) + \frac{6^s}{2\pi i\GG(s)} \int_{\frac12-i\infty}^{\frac12 + i\infty} \Gamma(s+z)\Gamma(-z) \zeta_{\mathrm{MT},2}(s,s-z,2s+z)dz\right).
	\end{multline*}
	Now, we have
	\begin{align*}
		\lim_{s\to\frac12} \left(s-\frac12\right)6^s\z_{\MT,2}(s,s,2s) = \frac{\sqrt{3}\pi}{2\sqrt{2}}.
	\end{align*}
	On the other hand, we find
	\begin{multline}\label{eq:str}
		\lim_{s\to\frac12} \left(s-\frac12\right)\frac{6^s}{2\pi i\GG(s)}\int_{\frac12-i\infty}^{\frac12+i\infty} \GG(s+z)\GG(-z)\z_{\MT,2}(s,s-z,2s+z) dz\\
		= \lim_{s\to\frac12} \left(s-\frac12\right)\frac{6^s\GG(3s-1)\z(4s-1)}{2\pi i\GG(s)}\int_{\frac12-i\infty}^{\frac12+i\infty} \GG(s+z)\GG(-z)\GG(z+1-s) dz,
	\end{multline}
	since $s\mapsto\frac{\GG(s+z)\GG(-z)\z(3s+z)\z(s-z)}{\GG(s)}$ and $s\mapsto\frac{\GG(s+z)\GG(-z)I_1(s;z)}{\GG(s)}$ are holomorphic if $\re(z)=\frac12$. Shifting the path to the left and using \cite[9.113]{GradRyzh}, \Cref{paper:GammaCollect} \ref{paper:GC:1}, 15.4.26 of \cite{NIST}, and Proposition \ref{paper:GammaCollect} \ref{paper:GC:5} we obtain that \eqref{eq:str} equals%and collecting the residues (see \cite[9.113]{GradRyzh}) that \eqref{eq:str} equals
	\[
		\frac{\sqrt3\pi}{2\sqrt2}{}_2F_1\left(\frac12,\frac12;1;-1\right) - \frac{\sqrt3\pi}{2\sqrt2} = \frac{\sqrt3\GG\left(\frac14\right)^2}{8\sqrt\pi} - \frac{\sqrt3\pi}{2\sqrt2}.
	\]
	This proves the first part of the proposition.
	
	Now, let $d\in\Z_{\le1}\sm(-3\N_0)$ and choose $0<\e<\frac13$, and also $K,M>1-d$. We have, by \eqref{paper:zeta0:1},
	\begin{multline}\label{paper:eq:zetaRes13comp:1}
		\Res_{s=\frac d3} \z_{\SO(5)}(s) = \lim_{s\to\frac d3} \frac{\left(s-\frac d3\right)6^s}{\GG(s)}\sum_{k=0}^{K-1} \frac{(-1)^k\GG(s+k)}{k!}\z_{\MT,2}(s,s-k,2s+k)\\
		+ \lim_{s\to\frac d3} \frac{\left(s-\frac d3\right)6^s}{2\pi i\GG(s)}\int_{K-\e-i\infty}^{K-\e+i\infty} \GG(s+z)\GG(-z)\z_{\MT,2}(s,s-z,2s+z) dz.
	\end{multline}
	Note that $\lim\limits_{s\to\frac d3}(s-\frac d3)I_M(s;k)=0$ because of holomorphicity of $I_M$ by \Cref{paper:L:Ibound} and
	\begin{equation*}%\label{paper:eq:limresn3}
		\lim_{s\to\frac d3} \left(s-\frac d3\right)\z(3s+k+m) = \frac13\d_{m=1-d-k}.
	\end{equation*}
	Thus we obtain, by \eqref{paper:eq:MTZeta-Formula} and (15.4.26) of \cite{NIST},
	\begin{align}\nonumber
		&\lim_{s\to\frac d3} \frac{\left(s-\frac d3\right)6^s}{\GG(s)}\sum_{k=0}^{K-1} \frac{(-1)^k\GG(k+s)}{k!}\z_{\MT,2}(s,s-k,k+2s) = \frac{6^\frac d3\z\left(\frac{4d}{3}-1\right)}{3(1-d)!\GG\left(\frac d3\right)}\\
		\nonumber
		&\hspace{1cm}\times \left(\sum_{k=0}^{K-1} \frac{(-1)^{k+d+1}\GG\left(k+1-\frac d3\right)\GG\left(k+\frac d3\right)}{k!\GG\left(k+\frac{2d}{3}\right)} + \sum_{k=0}^{1-d} (-1)^k\binom{1-d}{k}\frac{\GG\left(k+\frac d3\right)\GG\left(1-\frac{2d}{3} -k\right)}{\GG\left(\frac d3\right)}\right)\\
		\nonumber
		&\hspace{.2cm}= \frac{6^\frac d3\z\left(\frac{4d}{3}-1\right)}{3(1-d)!\GG\left(\frac d3\right)}\left(\sum_{k=0}^{K-1} \frac{(-1)^{k+d+1}\GG\left(k+1-\frac d3\right)\GG\left(k+\frac d3\right)}{k!\GG\left(k+\frac{2d}{3}\right)} + \GG\left(1-\tfrac{2d}{3}\right){}_2F_1\left(\tfrac d3,d-1;\tfrac{2d}{3};-1\right)\right)\\
		\nonumber
		&\hspace{.2cm}= \frac{6^\frac d3\z\left(\frac{4d}{3}-1\right)}{3(1-d)!\GG\left(\frac d3\right)}\sum_{k=0}^{K-1} \frac{(-1)^{k+d+1}\GG\left(k+1-\frac d3\right)\GG\left(k+\frac d3\right)}{k!\GG\left(k+\frac{2d}{3}\right)}\\
		\label{paper:eq:zetaRes13comp:2}
		&\hspace{7cm}+ \frac{3^{\frac d3-1}\z\left(\frac{4d}{3}-1\right)\GG\left(1-\frac{2d}{3}\right)\GG\left(\frac{2d}{3}\right)\GG\left(\frac d6\right)}{2^\frac d3(1-d)!\GG\left(\frac d3\right)^2\GG\left(\frac d2\right)}.
	\end{align}
	For the integral in \eqref{paper:eq:zetaRes13comp:1}, we obtain that
	\begin{multline} \label{paper:eq:limn3integral}
		\lim_{s\to\frac d3} \frac{\left(s-\frac d3\right)6^s}{2\pi i\GG(s)}\int_{K-\e-i\infty}^{K-\e+i\infty} \GG(s+z)\GG(-z)\z_{\MT,2}(s,s-z,2s+z) dz\\
		= \frac{(-1)^{d+1}6^\frac d3\z\left(\frac{4d}{3}-1\right)}{3(1-d)!\GG\left(\frac d3\right)}\frac{1}{2\pi i}\int_{K-\e-i\infty}^{K-\e+i\infty} \frac{\GG\left(z+\frac d3\right)\GG\left(z+1-\frac d3\right)\GG(-z)}{\GG\left(z+\frac{2d}{3}\right)} dz.
	\end{multline}
	By shifting the path of integration to the left such that all poles of $\GG(\frac d3+z)\GG(1-\frac d3+z)\GG(-z)$ except the ones in $\N_0$ lie left to the path of integration, we obtain with formula (9.113) of \cite{GradRyzh}
	\begin{align*}
		&\frac{1}{2\pi i}\int_{K-\e-i\infty}^{K-\e+i\infty} \frac{\GG\left(z-\frac d3\right)\GG\left(z+1-\frac d3\right)\GG(-z)}{\GG\left(z+\frac{2d}{3}\right)} dz\\
		&\hspace{2.5cm}= \frac{\GG\left(\frac d3\right)\GG\left(1-\frac d3\right)}{\GG\left(\frac{2d}{3}\right)}{}_2F_1\left(\frac d3,1-\frac d3;\frac{2d}{3};-1\right) + \sum_{k=0}^{K-1} \frac{(-1)^{k+1}\GG\left(k+\frac d3\right)\GG\left(k+1-\frac d3\right)}{k!\GG\left(k+\frac{2d}{3}\right)}\\
		%\label{paper:eq:limn3integral2} 
		&\hspace{2.5cm}= \frac{\GG\left(1-\frac d3\right)\GG\left(\frac d6\right)}{2\GG\left(\frac d2\right)} - \sum_{k=0}^{K-1} \frac{(-1)^k\GG\left(k+\frac d3\right)\GG\left(k+1-\frac d3\right)}{k!\GG\left(k+\frac{2d}{3}\right)},
	\end{align*}
	where the final equality is due to (15.4.26) of \cite{NIST}. \Cref{paper:resu} follows by this calculation together with \eqref{paper:eq:zetaRes13comp:1}, \eqref{paper:eq:zetaRes13comp:2}, \eqref{paper:eq:limn3integral}, and \Cref{paper:GammaCollect} \ref{paper:GC:5}. Finally note that \eqref{paper:resu} vanishes for even $d\le1$.
\end{proof}

Now we are ready to prove Theorem \ref{paper:T:rsu5asy}.

\begin{proof}[Proof of Theorem \ref{paper:T:rsu5asy}]
	Note that by \Cref{paper:lem:SO5conditions} and \Cref{paper:T:mainSO5} all conditions of \Cref{paper:T:main2} are satisfied (with $L$ and $R\notin\frac13\N$ arbitrary large). As $\z_{\SO(5)}$ has, by \Cref{paper:P:reszetaso(5)1/2}, exactly two positive poles $\a:=\frac12>\frac13=:\b$, \Cref{paper:T:TwoPoleAsymptotics} applies with $\ell=3$, and we obtain
	\[
		r_{\SO(5)}(n) = \frac{C}{n^b}\exp\left(A_1n^\frac13+A_2n^\frac29+A_3n^\frac19+A_4\right)\left(1+\sum_{j=2}^{N+1} \frac{B_j}{n^\frac{j-1}{9}}+O_N\left(n^{-\frac{N+1}{9}}\right)\right),\quad (n\to\infty).
	\]
	So we are left to calculate $c$, $b$, $A_1$, $A_2$, $A_3$, and $A_4$. By \Cref{paper:P:zetaso5(0)}, $\z_{\SO(5)}(0)=\frac38$ and by \Cref{paper:P:reszetaso(5)1/2}, $\Res_{s=\frac12}\z_{\SO(5)}(s)$, $\w_\frac12=\frac{\sqrt3\GG(\frac14)^2}{8\sqrt\pi}$ and $\w_\frac13=\frac{2^\frac13+1}{3^\frac23}\z(\frac13)$. Hence, by \eqref{paper:eq:cDef}, we get
	\[
		c_1\frac{\sqrt3\GG\left(\frac14\right)^2\z\left(\frac32\right)}{16},\qquad c_2 = 3^{-\frac53}\left(2^\frac13+1\right)\GG\left(\frac13\right)\z\left(\frac13\right) \z\left(\frac43\right).
	\]
	Moreover, by \Cref{paper:L:varrho}, we have
	\begin{align*}
		K_2 = \frac{2c_2}{3c_1^{\frac{2}{9}}},\quad K_3 = -\frac{c_2^2}{27c_1^{\frac{10}{9}}}.
	\end{align*}
	Now, we compute $A_1$, $C$, and $b$ by \eqref{paper:eq:mainconstants2} and $A_2$, $A_3$, $A_4$ by \Cref{paper:T:TwoPoleAsymptotics} and obtain
	\begin{align}\label{paper:A}
		b &= \frac{7}{12},\qquad C = \frac{e^{\z_{\SO(5)}'(0)}\GG\left(\frac14\right)^\frac16 \z\left(\frac32\right)^\frac{1}{12}}{2^\frac133^\frac{11}{24}\sqrt\pi},\qquad A_1 = \frac{3^\frac43\GG\left(\frac14\right)^\frac43\z\left(\frac32\right)^\frac23}{2^\frac83},\\
		A_2 &= \frac{2^\frac89\left(2^\frac13+1\right)\GG\left(\frac13\right)\z\left(\frac13\right) \z\left(\frac43\right)}{3^\frac79\GG\left(\frac14\right)^\frac49 \z\left(\frac32\right)^\frac29},\qquad A_3 = -\frac{2^\frac{40}{9}\left(2^\frac13+1\right)^2\GG\left(\frac13\right)^2 \z\left(\frac13\right)^2\z\left(\frac43\right)^2}{3^\frac{44}{9} \GG\left(\frac14\right)^\frac{20}{9}\z\left(\frac32\right)^\frac{10}{9}},\\
		\label{paper:A2}
		A_4 &= \frac{2^8\left(2^\frac13+1\right)^3\GG\left(\frac13\right)^3\z\left(\frac13\right)^3 \z\left(\frac43\right)^3}{3^8\GG\left(\frac14\right)^4\z\left(\frac32\right)^2}.
	\end{align}
	This proves the theorem.
\end{proof}

\section{Open problems}\label{sec:open}

We are led by our work to the following questions:
\begin{enumerate}[leftmargin=19pt, labelwidth=!, labelindent=0pt]
	\item Is there a simple expression for $\zeta'_{\so(5)}(0)$?
	
	\item Can one weaken the hypothesis that $f(n) \geq 0$ for all $n$ in Theorem \ref{paper:T:main2}? An important application would be that the $r_f(n)$ are eventually positive. There are many special cases in the literature (see \cite{Chern,C1,C2,Craig}), but to the best of our knowledge no general asymptotic formula has been proved.\footnote{The one exception is in Todt's Ph.D. thesis \cite[Theorem 3.2.1]{Todt}; however, there it is further assumed that $r_f(n)$ is non-decreasing, which precludes the princple application of such an asymptotic.}
	
	\item In \cite{Erdos}, Erd\H{o}s proved by elementary means that if $S\subset\N$ has natural density $d$ and $\1_S$ is the indicator function of $S$, then $\log(p_{\1_S}(n))\sim\pi\sqrt\frac{2dn}{3}$. Referring to Theorem \ref{paper:T:main2}, can one prove by elementary means that for any $\e>0$
	$$
		\log \left(r_f(n)\right)=A_1n^{\frac{\alpha}{\alpha+1}}+\sum_{j=2}^MA_jn^{\alpha_j}+O(n^{\varepsilon})?
	$$
	
	\item Can one ``twist'' the products in \Cref{paper:T:main2} by $w\in\C$ and prove asymptotic formulas for the (complex) coefficients of
	\[
		\prod_{n\ge1} \frac{1}{\left(1-wq^n\right)^{f(n)}} ?
	\]
	If $f(n)=n$ or $f(n)=1$, then such asymptotics were shown to determine zero attractors of polynomials (see \cite{BG,BP}) and equidistribution of partition statistics see \cite{BFG,BFM}), and the general case of $|w|\ne1$ was treated by Parry \cite{P}. Nevertheless, all of these results require that $L_f(s)$ has only a single simple pole with positive real part.
	
	\item In \Cref{paper:T:main2}, can one write down explicit or recursive expressions for the constants $A_j$ in the exponent, say in the case that $L_f(s)$ has three positive poles?
	
	\item Can one prove limit shapes for the partitions generated by \eqref{def:GfLf} in the sense of \cite{DVZ,V}?
\end{enumerate}


\begin{thebibliography}{99}
	%\bibitem{Amann} H. Amann, J. Escher, \textit{Analysis 1}, Dritte Auflage, Birkh\"auser, 2006.
	
	%\bibitem{Andrews} G. Andrews, \textit{The theory of partitions}, Cambridge University Press, 1998.
	
	\bibitem{AAR} G. Andrews, R. Askey, and R. Roy, {\it Special functions}, Encyclopedia of Mathematics and its Applications {\bf 71}, Cambridge University Press, 2000.
	
	\bibitem{Apostol} T. Apostol, {\it Introduction to analytic number theory}, Springer Science + Business Media, 1976.
	
	\bibitem{BG} R. Boyer and W. Goh, {\it Partition polynomials: asymptotics and zeros,} arXiv:0711.1373.
	
	\bibitem{BP} R. Boyer and D. Parry, {\it On the zeros of plane partition polynomials}, Electron. J. Combin. {\bf 18} (2012)
	
	\bibitem{BFG} W. Bridges, J. Franke, and T. Garnowski, {\it Asymptotics for the twisted eta-product and applications to sign changes in partitions}, Res. Math. Sci. {\bf9} (2022)
	
	\bibitem{BFM} W. Bridges, J. Franke, and J. Males, {\it Sign changes in statistics for plane partitions,} arXiv:2207.14590
	
	\bibitem{BF} K. Bringmann and J. Franke, {\it An asymptotic formula for the number of $n$-dimensional representations of $\SU(3)$}, Revista Matemática Iberoamericana, accepted for publication. 
	
	%\bibitem{BSM} K. Bringmann, C. Jennings-Shaffer, and K. Mahlburg, \textit{On a Tauberian theorem of Ingham and Euler--Maclaurin summation}, Ramanujan Journal, 2021. 
	
	\bibitem{Brue} J. Br\"udern, {\it Einf\"uhrung in die analytische Zahlentheorie}, Springer, 1995.
	
	\bibitem{Burckel} R. Burckel, {\it Classical analysis in the complex plane}, Birkh\"auser, 2021. 
		
	%\bibitem{Cauchy} A.-L. Cauchy, {\it Démonstration du théorèm général de Fermat sur les nombres polygones}, Mém. Sci. Math. Phys. Inst. France {\bf 14} (1813-1815), 177–220; Oeuvres complètes {\bf VI} (1905), 320–353. 
	
	\bibitem{Char} C. Charalambides, {\it Enumerative combinatorics}, Chapman and Hall/CRC, 2002.
	
	\bibitem{Chern} S. Chern, {\it Nonmodular infinite products and a conjecture of Seo and Yee}, arXiv:1912.10341.
	
	\bibitem{C1} A. Ciolan, {\it Asymptotics and inequalities for partitions into squares}, Int. J. Number Theory {\bf16} (2020), 121--143.
		
	\bibitem{C2} A. Ciolan, {\it Equidistribtion and inequalities for partitions into powers}, arXiv:2002.05682
	
	\bibitem{Craig} W. Craig, {\it Seaweed algebras and the index statistic for partitions}, arXiv:2112.09269.
		
	% \bibitem{Conway} J. Conway, {\it Functions of One Complex Variable I}, Second Edition, Springer, 1978. 
	
	\bibitem{DebrTen} G. Debruyne and G. Tenenbaum, {\it The saddle-point method for general partition functions}, Indagationes Mathematicae {\bf31} (2020), 728--738.
	
	\bibitem{DVZ} A. Dembo, A. Vershik, and O. Zeitouni, {\it Large deviations for integer partitions}, Markov Process. Related Fields {\bf 6} (2000), 147--179.
	
	\bibitem{DunnRob} A. Dunn and N. Robles, {\it Polynomial partition asymptotics}, Journal of Mathematical Analysis and Applications 
	{ \bf 459}
	(2018),
359--384.
	
	\bibitem{Erdos} P. Erd\H{o}s, {\it On an elementary proof of some asymptotic formulas in the theory of partitions,} Annals of Math. {\bf 43} (1942), 437--450.
	
	%\bibitem{DR} A. Dunn and N. Robles, \textit{Polynomial partition asymptotics}, Journal of Mathematical Analysis and Applications \textbf{459}, 359--384.
	
	% \bibitem{Duren} P. Duren, {\it Univalent Functions}, Springer Verlag, 1983. 
	
	%\bibitem{Flajolet} P. Flajolet and R. Sedgewick, {\it Analytic Combinatorics}, Cambridge University Press, 2009. 
	
	%\bibitem{Ga} A. Gafni, \textit{Power partitions}, Journal of Number Theory {\bf 163} (2016), 19--42.
	
	\bibitem{GradRyzh} I. Gradshteyn and I. Ryzhik, {\it Table of integrals, series, and products}, Sixth Edition, Academic Press, 2000.
	
	\bibitem{Hall} B. Hall, {\it Lie groups, lie algebras, and representations: An elementary introduction.} Springer, 2004.
	
	\bibitem{HardyRama} G. Hardy and S. Ramanujan, {\it Asymptotic formulae in combinatory analysis}, Proceedings of the London Mathematical Society (2) {\bf 17} (1918), 75–115.
	
	%\bibitem{Kaup} L. Kaup and B. Kaup, {\it Holomorphic functions of several variables}, de Gruyter Studies in Mathematics {\bf 3}, 1983.
		
	\bibitem{KelMas} J. Kelliher and R. Masr, {\it Analytic continuation of multiple zeta functions}, Cambridge University Press {\bf 3}, 2008.
	
	%\bibitem{Lang} S. Lang, {\it Complex Analysis}, Fourth Edition, Springer, 1999.
	
	%\bibitem{LangAna} S. Lang, {\it Real and Functional Analysis}, Third Edition, Springer, 1993.
		
	%\bibitem{LangUGA} S. Lang, {\it Undergraduate Analysis}, Second Edition, Springer, 1997.
		
	\bibitem{MacMahon} P. MacMahon, {\it Combinatory analysis}, Vol. I, II, Dover Publ., Mineola, New York, 2004.
	
	% \bibitem{Mah} K. Mahler, \textit{\"Uber einen Satz von Mellin}, Math. Ann. {\bf 100} (1928), 384--398.
	
	\bibitem{Mat2} K. Matsumoto, {\it On analytic continuation of various multiple zeta-functions}, in Number Theory for the Millennium (Urbana, 2000), Vol. II, M. Bennett et. al. (eds.), A. Peters, Natick, MA, (2002), 417--440.
	
	\bibitem{Mat} K. Matsumoto, {\it On Mordell--Tornheim and other multiple zeta-functions}, in Proceedings of the Session in Analytic Number Theory and Diophantine Equations, D. Heath-Brown and B. Moroz (eds.), Bonner Math. Schriften 360, Univ. Bonn, Bonn, 2003, {\bf25}.
	
	\bibitem{MatTsu} K. Matsumoto and H. Tsumura, {\it On Witten multiple Zeta-functions associated with semisimple Lie algebras I}, Ann. Inst. Fourier, Grenoble {\bf5} (2006), 1457--1504.
	
	\bibitem{MatWen} K. Matsumoto and L. Weng, \textit{Zeta-functions defined by two polynomials}, Proceedings of the 2nd China-Japan Seminar (Iizuka, 2001), Developments in Math. {\bf 8}, S. Kanemitsu, C. Jia (eds.), Kluwer Academic Publishers, 2002, 233--262.
	
	\bibitem{Meinardus} G. Meinardus, {\it Asymptotische Aussagen \"uber Partitionen}, Math. Z. {\bf59}, 1954, 388--398. 
	
	% \bibitem{MSV} J. Mehta, B. Saha, and G. Viswanadham, \textit{Analytic properties of multiple zeta functions and certain weighted variants, an elementary approach}, Journal of Number Theory {\bf 168} (2016), 487--508.
	
	%	\bibitem{Morse} P. Morse and H. Feshbach, {\it Methods of Theoretical Physics, Part I.} New York: McGraw-Hill, 1953.
	
	% \bibitem{Overholt} M. Overholt, {\it A Course in Analytic Number Theory}, AMS, Volume 160, 2014. 
	
	%\bibitem{Remmert} R. Remmert, G. Schumacher, {\it Funktionentheorie 2}, 3. Auflage, Springer, 2007. 
	% \bibitem{Palka} B. Palka, {\it An Introduction to Complex Function Theory}, Springer, 1991.
	
	\bibitem{NIST} F. Olver, D. Lozier, R. Boisvert and C. Clark, \textit{The NIST Handbook of Mathematical Functions}, Cambridge University Press, New York, NY, (2010).
	
	\bibitem{P} D. Parry, {\it A polynomial variation on Meinardus’ theorem}, Int. J. of Number Theory {\bf11} (2015), 251--268.
	
	\bibitem{Ro} D. Romik, {\it On the number of $n$-dimensional representations of $\SU(3)$, the Bernoulli numbers, and the Witten zeta function}, arXiv:1503.03776
	
	% \bibitem{Stein} E. Stein and R. Shakarchi, {\it Complex Analysis}, Princeton Lectures in Analysis, Princeton University Press, 2003.
	
	% \bibitem{Temme} N. Temme, {\it Asymptotic methods for integrals}, Series in Analysis Vol. {\bf 6}, World Scientific, 2015.
	
	\bibitem{Tenenbaum} G. Tenenbaum, {\it Introduction to analytic and probabilistic number theory}, Graduate Studies in Mathematics, American Mathematical Society \textbf{163}, third edition, 2008.
	
	\bibitem{Todt} H. Todt, {\it Asymptotics of partition functions,} Ph.D. Thesis, Pennsylvania State University, 2011.
	
	%\bibitem{TWL} G. Tenenbaum, J. Wu, and Y. Li, \textit{Power partitions and Saddle-point method}.
	
	\bibitem{V} A. Vershik, {\it Statistical mechanics of combinatorial partitions, and their limit configurations}, Translation from Russian, in: Funct. Anal. Appl. {\bf 30} (1996), 90--105.
	
	\bibitem{Wright} E. Wright, {\it Asymptotic partition formulae I: Plane partitions}. Q. J. Math. {\bf 1} (1931), 177–189.
	
	%\bibitem{Wr} M. Wright, \textit{Asymptotic partition formulas III, partition into $k$-th powers.}
	
	% \bibitem{Zag} D. Zagier, {\it Values of zeta functions and their applications}, in Proc. First Congress of Math., Paris, vol.II, Progress in Math. {\bf 120}, Birkh\"auser, (1994), 497--512.
	
	%\bibitem{ABMP} S.~Alexandrov, S.~Banerjee, J.~Manschot, and B.~Pioline, {\it Indefinite theta series and generalized error functions}, Selecta Mathematica (NS) {\bf 24} (2018), 3927.
	
	%\bibitem{BKMZ} K. Bringmann, J. Kaszian, A. {Milas,} and S. Zwegers, \emph{Rank two false theta functions and Jacobi forms of negative definite matrix index}, {Advances in Applied Mathematics} {\bf 112} (2020), 101946.
	
	%\bibitem{BMM} K. Bringmann, K. Mahlburg and A. {Milas}, {\em Higher depth quantum modular forms and plumbed $3$-manifolds}, to appear in Letters in Mathematical Physics, arXiv:1906.10722.
	
	%\bibitem{BN} K.~Bringmann and C.~Nazaroglu, {\it A Framework for Modular Properties of False Theta Functions}, Research in the Mathematical Sciences {\bf 6} (2019).
	
	%\bibitem{BNi} M. Buican and T. Nishinaka, \emph{On the superconformal index of Argyres-
	%Douglas theories}, Journal of Physics A {\bf 49} (2016), 015401.

	%\bibitem{CS} C. Cordova and S.-H. Shao, {\em Schur Indices, BPS Particles, and Argyres-Douglas Theories}, {Journal of High Energy Physics} {\bf 2016} (2016), 40.
	
	%\bibitem{Creutzig0} T. Creutzig, \emph{W-algebras for Argyres-Douglas theories,} European Journal of Mathematics {\bf 3.3} (2017), 659-690.
	
	%\bibitem{Gukov} S. Gukov, D. Pei, P. Putrov, and C. Vafa, {\em BPS spectra and 3-manifold invariants}, \texttt{arXiv:1701.06567}.
	
	%\bibitem{GM} S. Gukov and C. Manolescu, {\em A two-variable series for knot complements}, arXiv:1904.06057 (2019).
	
	%\bibitem{KW0} V. Kac and M. Wakimoto, \emph{A remark on boundary level admissible representations}, Comptes Rendus Mathematique {\bf 355.2} (2017), 128-132.
\end{thebibliography}
\end{document}